\documentclass[12pt]{article}
\usepackage{geometry}
\usepackage{amsthm,amsmath,amsfonts,amssymb,amscd,mathtools}
\usepackage{graphics,caption,fancyhdr,float,xcolor,hyperref}
\usepackage{appendix,enumerate,latexsym,extarrows}
\usepackage[nottoc]{tocbibind}  

\usepackage[scr]{rsfso}
\usepackage{tikz-cd}
\allowdisplaybreaks

\newtheorem{thm}{Theorem}[section]
\newtheorem{conj}[thm]{Conjecture}
\newtheorem{lem}[thm]{Lemma}
\newtheorem{prop}[thm]{Proposition}
\newtheorem{crl}[thm]{Corollary}
\theoremstyle{definition}
\newtheorem{rem}[thm]{Remark}

\newtheorem{df}[thm]{Definition}

\newtheorem{example}[thm]{Example}

\newcommand{\ov}{\overline}
\newcommand{\lan}{\langle}
\newcommand{\ran}{\rangle}
\newcommand{\wt}{\widetilde}
\newcommand{\ot}{\otimes}
\newcommand{\hra}{\hookrightarrow}
\newcommand{\<}{\langle}
\renewcommand{\>}{\rangle}
\newcommand{\N}{\mathbb{N}}
\newcommand{\Z}{\mathbb{Z}}
\newcommand{\Zt}{\Z/2\Z}

\newcommand{\C}{\mathbb{C}}
\newcommand{\A}{\mathbb{A}}
\newcommand{\Gm}{\C^*}

\newcommand{\SG}{{\!\mathsf{SG}}}
\newcommand{\MF}{\mathrm{MF}}
\newcommand{\PV}{\mathsf{\!MF}}  
\newcommand{\PVs }{MF\ }  
\newcommand{\Ch}{\mathrm{Ch}}

\newcommand{\cC}{\mathcal{C}}
\newcommand{\cF}{\mathscr{F}}
\newcommand{\cH}{\mathscr{H}}
\newcommand{\cHw}{\cH(\bw,G_\bw)}
\newcommand{\cJ}{\mathscr{J}}

\newcommand{\cL}{\mathcal{L}}
\newcommand{\cM}{\mathscr{M}}
\newcommand{\oM}{\overline{\cM}}
\newcommand{\cN}{\mathcal{N}}
\newcommand{\cO}{\mathcal{O}}
\newcommand{\cQ}{\mathscr{Q}}
\newcommand{\cS}{\mathcal{S}}
\newcommand{\cSr}{\cS}

\newcommand{\bK}{{\boldsymbol{K}}}
\newcommand{\bl}{{\boldsymbol{\ell}}}
\newcommand{\bm}{{\boldsymbol{m}}}
\newcommand{\bo}{{\boldsymbol{\omega}}}
\newcommand{\bp}{{\boldsymbol{p}}}
\newcommand{\bq}{{\boldsymbol{q}}}
\newcommand{\brho}{{\boldsymbol{\rho}}}
\newcommand{\bs}{{\boldsymbol{s}}}
\newcommand{\bt}{{\boldsymbol{t}}}
\newcommand{\bw}{{\boldsymbol{w}}}

\newcommand{\bxi}{{\boldsymbol{\xi}}}

\newcommand{\fC}{\mathfrak{C}}
\newcommand{\fF}{\mathfrak{F}}
\newcommand{\fI}{\mathfrak{I}}
\newcommand{\fS}{\mathfrak{B}}
\newcommand{\de}{\delta}
\newcommand{\Ga}{\Gamma}
\newcommand{\ga}{\gamma}
\newcommand{\gb}{{\boldsymbol\ga}} 
\newcommand{\ogamma}{\overline{\gb}}
\renewcommand{\th}{\theta}
\newcommand{\bth}{{\boldsymbol{\theta}}}
\newcommand{\one}{{\boldsymbol{1}}}
\newcommand{\zb}{\boldsymbol{\zeta}}

\newcommand{\id}{\operatorname{id}}
\newcommand{\Res}{{\operatorname{Res}}}
\newcommand{\soc}{\mathrm{soc}}
\newcommand{\HH}{\operatorname{HH}}

\begin{document}
\title
{A Landau-Ginzburg mirror theorem via matrix factorizations}
\date{}
\author{Weiqiang He, Alexander Polishchuk, Yefeng Shen and Arkady Vaintrob}

\maketitle

\begin{abstract}
For an invertible quasihomogeneous
polynomial  $\bw$ we prove an all-genus mirror theorem
relating two cohomological field theories of the Landau-Ginzburg type.
On the $B$-side it is the Saito-Givental theory for a specific
choice of a primitive form. On the $A$-side, it is the matrix factorization CohFT
for the dual singularity $\bw^T$ with the maximal diagonal symmetry group.
\end{abstract}

{\hypersetup{linkcolor=black}
\setcounter{tocdepth}{2}
\tableofcontents
}

\section{Introduction}
Mirror symmetry, which started with a discovery by theoretical physicists that different
geometric inputs can produce equivalent string theory models, brought spectacular
developments in mathematics. In particular, it predicted that Gromov-Witten invariants of
a Calabi-Yau manifold $X$  (observables in the $A$-model topological strings  on $X$)
depending only on the symplectic structure of $X$ can be expressed in terms of the
$B$-model observables for another Calabi-Yau manifold $Y$ which depend on the complex
structure of $Y$.

To describe the formal structure  of Gromov-Witten invariants, Kontsevich and Manin~\cite{KM}
introduced the notion of a \emph{cohomological field theory} (CohFT) which is a
finite-dimensional vector space $\cH$ (the \emph{state space} of the theory) with a
nondegenerate symmetric pairing and a collection of operations
\begin{equation}
  \Lambda_{g,r}: \cH^{\ot r}\to H^{*}(\oM_{g,r})
\end{equation}
with values in the cohomology of the Deligne-Mumford moduli space $\oM_{g,r}$ of stable curves of
genus $g$ with $r$ marked points. These operations satisfy some natural factorization axioms
(see  Sec.\ \ref{sec:cohft}).
A CohFT is a very rich structure. A substantial part of it
is encoded in numerical invariants called the \emph{correlators}
(see~\eqref{eq:correlators}) which are obtained by intersecting the classes
$\Lambda_{g,r}(h_1,\ldots,h_r)$ with some tautological classes on $\oM_{g,r}$.
In particular, the $g=0, r=3$ correlators turn the state space $\cH$ into a Frobenius
algebra and the collection of all genus-zero correlators equips it with a structure of a
formal Frobenius manifold.

This formalism allows to give a mathematical interpretation of mirror symmetry as a
statement about isomorphism between two CohFTs constructed from different geometric data.
Most of the known mirror symmetry results compare only parts of the structures of the
corresponding CohFTs. For example, we can speak of mirror symmetry at the state space
level, or of Frobenius algebra isomorphisms, or of a genus-zero mirror symmetry.
Complete mirror symmetry results, valid for all genera, are very rare.
In this paper we prove such a theorem for CohFTs corresponding to Landau-Ginzburg models
coming from invertible singularities.

Besides topological strings, another common source of CohFTs
are Landau-Ginzburg models whose geometric input is a holomorphic function
$$\bw:\C^n\to \C$$
with an isolated singularity at the origin.
There are numerous examples of mirror symmetry phenomena involving Landau-Ginzburg
models and theories of Gromov-Witten type.
Most of them deal with the Landau-Ginzburg $B$-model which in various forms existed since 1990s.
The operations~\eqref{eq:cohft} of the corresponding CohFT are constructed using Saito's
theory of primitive forms~\cite{Sai} and Givental's quantization  procedure~\cite{Giv}.

The first mathematical theory of a Landau-Ginzburg $A$-model was constructed much
later by Fan, Jarvis and Ruan~\cite{FJR07,FJR}.
Based on an earlier idea of Witten~\cite{Wit},
these CohFTs became known as FJRW theories.
Their construction
paved the way for a mathematical study of mirror symmetry between different Landau-Ginzburg models
whose existence was earlier suggested by physicists.
An FJRW theory depends on a choice of a \emph{nondegenerate} quasihomogeneous polynomial
function  $\bw:\C^n\to \C$ with an isolated singularity at the origin and an \emph{admissible}
group $G$ of diagonal symmetries of $\bw$
(see Section~\ref{sec:lg-mirror} for details).
An LG/LG mirror symmetry starts with an
 \emph{invertible polynomial}
 \begin{equation}
   \label{eq:invertible_poly}
\bw=\sum_{i=1}^n\prod_{j=1}^nx_j^{a_{ij}}
 \end{equation}
(a nondegenerate quasihomogeneous polynomial on $\C^n$ with exactly $n$ nonzero monomials)
determined by the \emph{exponent matrix}
\begin{equation}  \label{eq:exp-matrix}
E_\bw=(a_{ij}).
\end{equation}
The mirror partner of $\bw$ introduced by Berglund and H\"ubsch~\cite{BH} is the \emph{dual polynomial}
\begin{equation}
 \label{eq:dual_poly}
\bw^T:=\sum_{i=1}^n\prod_{j=1}^nx_j^{a_{ji}}
\end{equation}
with the transposed exponent matrix $E_{\bw^T}=E_\bw^T$.
Later Berglund and Henningson~\cite{BHe} extended this construction to include
admissible groups of symmetries and provided initial evidence supporting the conjecture that the $A$
and $B$ models for  the dual LG pairs $(\bw,G)$ and $(\bw^T,G^T)$ are equivalent.
In~\cite{Kr9,Kr10}  Krawitz established the base case of  this LG/LG mirror symmetry by constructing an
explicit isomorphism between the FJRW $A$-model state space
for $(\bw,G)$ and the $B$-model state space  for the dual pair $(\bw^T,G^T)$.

The FJRW CohFT has been constructed for all admissible pairs $(\bw,G)$, but its $B$-side
counterpart, the Saito-Givental theory, currently is known only for pairs $(\bw^T,G^T)$
with the trivial group $G^T$, i.e.\ when $G=G_\bw$ is the
maximal diagonal symmetry group of $\bw$. Thus, to extend the LG/LG mirror symmetry beyond the level
of Frobenius algebras, we must restrict ourselves to the case $G=G_\bw$.
Over the last two decades this correspondence has been established for several families of
invertible polynomials, starting with the singularities of type $A$~\cite{JKV, FSZ} (in
which case it is equivalent to the generalized Witten's conjecture~\cite{Wit}), then
continuing with singularities of the types $D$ and $E$~\cite{FJR}, simple elliptic
singularities~\cite{KSh, MS}, exceptional unimodal singularities~\cite{LLSS}, and
culminating in the work~\cite{HLSW} which proved LG/LG mirror symmetry for almost all
invertible polynomials.
\begin{thm}[\sf FJRW-SG mirror symmetry~{\cite[Theorem 1.2]{HLSW}}]
 \label{thm-fjrw}
 Let $\bw$ be an invertible polynomial without chain variables of weight $1/2$.  Then
 there exists a primitive form $\zeta$ for the dual polynomial $\bw^T$ and an isomorphism
   \begin{equation}
     \label{eq:kraw-map}
     \th:  \cQ_{\bw^T} \to \cH(\bw,G_\bw),
       \end{equation}
between the Milnor algebra $\cQ_{\bw^T}$ of the singularity $\bw^T$ and the FJRW state space which
identifies the FJRW correlators for $(\bw,G_\bw)$ with the  corresponding correlators of the
Saito-Givental CohFT for the pair $(\bw^T,\zeta)$, for all $g$ and $r$.

In particular, the map $\th$ induces an isomorphism of the corresponding Frobenius manifolds.
 \end{thm}

The remaining cases are exactly those missing in Krawitz's theorem~\cite[Theorem 4.1]{Kr9} on mirror
symmetry for $(\bw,G_\bw)$ at the Frobenius algebra level,
when the invertible polynomial $\bw$ contains a chain summand of the form
\begin{equation}
\label{chain-weight-half}
\sum_{i=1}^{k-1}x_i^{a_i}x_{i+1}+x_k^2.
\end{equation}
Here $x_k$ is a chain variable of weight ${1\over 2}$.
The main difficulty here is that for such polynomials the FJRW state space
$\cH(\bw,G_\bw)$ contains so-called \emph{broad generators}
whose structure constants cannot be computed using the formal algebraic methods
of~\cite{FJR} or other available tools.
\\[4pt]

In this paper we use a different approach to circumvent this problem and to prove
an Landau-Ginzburg mirror symmetry theorem for all invertible polynomials without exceptions.

In~\cite{PV16} two of the current authors
gave a different construction of a
Landau-Ginzburg $A$-model CohFT with the same input as the FJRW theory.
The main technical tool of the construction is the categories of matrix factorizations and
for this reason we call the corresponding theory the \emph{matrix factorizations} (MF) CohFT.
The MF CohFT has the same state space  $\cH(\bw,G)$ as the FJRW CohFT,
and conjecturally the two theories are isomorphic.
However, because of the analytic difficulties of computing general FJRW
correlators, this conjecture has been verified only in some special cases
(see~\cite{CLL,Gue}).

The algebraic nature of the construction of the MF CohFT
makes it much more amenable for computations. This allowed us to overcome the
difficulties related to the existence of broad generators
and to prove a general Landau-Ginzburg mirror symmetry theorem valid for all invertible
polynomials and for all genera.

The key step of our proof is the following result establishing mirror symmetry
at the level of Frobenius algebras.

\begin{thm}[{\sf Mirror symmetry at the topological level}]
\label{frob-isom}
Let $\bw$ be an invertible polynomial.
There exists a linear map (the \emph{mirror map}
defined by explicit formulas in Definition~\ref{mirror-map})
\begin{equation}
  \label{eq:frob-iso}
  \th: \cQ_{\bw^T}\to \cH(\bw,G_\bw)
\end{equation}
which is an isomorphism of the Frobenius algebra structures on
the state space $\cH(\bw,G_\bw)$ of the MF CohFT and
on the Milnor ring  $\cQ_{\bw^T}$ of the dual polynomial $\bw^T$, the state space of
the Saito-Givental CohFT.
\end{thm}

Theorem~\ref{frob-isom} serves as a starting point for the proof of our main
result, which can be seen as an improvement of Theorem~\ref{thm-fjrw}.
\begin{thm}
[{\sf MF-SG Mirror Symmetry}]
\label{main-theorem}
Let $\bw$ be an invertible polynomial. Then for a specific choice of a
primitive form $\zeta$ for the dual polynomial $\bw^T$, the mirror map~\eqref{eq:frob-iso}
identifies the correlators~\eqref{eq:correlators} of
the MF CohFT
with the  corresponding correlators of the  Saito-Givental CohFT
for all $g$ and $r$.
\end{thm}
The proof of the theorem is derived from the mirror symmetry at the Frobenius
algebra level in two stages.
First we use the reconstruction techniques developed in~\cite{HLSW} and
computational tools
from~\cite{PV16, Gue} to prove that for a certain
primitive form $\zeta$ the map $\th$ is an isomorphism of Frobenius manifolds
for the two CohFTs.
Then we obtain the equality of $g>0$ correlators using the result of Milanov~\cite[Theorem
1.1]{Mi} that Givental-Teleman's~\cite{Giv, Tel} higher genus formulas for semisimple
Frobenius manifolds corresponding to isolated singularities uniquely extend to the
origin of the deformation space.

\subsection*{Plan of the paper}
This paper is organized as follows.
In Section~\ref{sec-prelim} we review the preliminaries and fix notation. We start by
reminding the basic notions related to CohFTs, isolated singularities and quasihomogeneous
polynomials.  Then we review the ingredients of two CohFTs related to singularities: the
Saito-Givental CohFT  and the matrix factorizations CohFT constructed in~\cite{PV16}.

In Section~\ref{sec-frob-alg} we compute the three-point correlators
$\langle \alpha,\beta,\gamma \rangle^\PV_{0,3}$  which determine the ring structure of
$\cH(\bw, G_\bw)$.
Then we construct the mirror map $\th$~\eqref{eq:frob-iso} and prove that it is an
isomorphism of Frobenius algebras, thus establishing
the mirror symmetry  at the topological
level (Theorem~\ref{frob-isom}).

In Section~\ref{sec-frob-mfd}, in Proposition~\ref{nonvanishing}, we find combinatorial
conditions for the nonvanishing of genus-zero $r$-point MF correlators.
Using it and the WDVV associativity relations we show in Proposition~\ref{reconstruction}
that the Frobenius manifold structure of the MF CohFT can be reconstructed from some special genus-zero four-point correlators.
Then we compute these correlators in Proposition~\ref{mirror-invariant-prop}
and use the results to identify the Frobenius manifolds of the two CohFTs.
This establishes the main result for $g=0$, which in turn implies our general mirror
symmetry Theorem~\ref{main-theorem}.

Our computations of the MF correlators
for invertible polynomials without weight $1/2$ chain variables match (up to a sign) the
ones performed in~\cite{HLSW}.
This, together with Krawitz's theorem~\cite[Theorem 4.1]{Kr9}, implies that for such polynomials
the FJRW and MF CohFT correlators agree for all $g$ and $r$ thus showing that our
Theorem~\ref{main-theorem} is an extension of Theorem~\ref{thm-fjrw}.

In the appendix~\ref{sec-appendix} we provide a
proof of the reconstruction theorem for polynomials of the chain type.
This proof is slightly shorter than a similar proof given in in~\cite{HLSW} and also provides the proof for the cases~\eqref{chain-weight-half} not considered there.

\subsection*{Acknowledgments}
W.H. and Y.S. would like to thank J\'er\'emy Gu\'er\'e,   Changzheng Li, Si Li,
Kyoji Saito, and Rachel Webb for helpful   discussions on mirror symmetry of LG models.
Y.S. thanks Marc Krawitz who, in a conversation around 2011, brought up an idea of using
matrix factorizations to study his mirror Frobenius algebras~\cite{Kr9}. We also thank
IHES (A.P.) and IMS, ShanghaiTech University (Y.S.) for hospitality and support.
A.P.\ was partially supported by NSF grant DMS-1700642 and Russian Academic Excellence Project 5-100. Y.S.\ was partially supported by the Simons Collaboration Grant 587119.

\section{Preliminaries}
\label{sec-prelim}
In this section we review some necessary facts and constructions about singularities of
functions and cohomological field theories of Landau-Ginzburg type.

\subsection{Cohomological field theories}
\label{sec:cohft}
Let $\oM_{g,r}$ be the Deligne-Mumford moduli space of stable curves of genus
$g$ with $r$ marked points.
Recall that a  ($\C$-valued) \emph{cohomological field theory  (CohFT) with a unit}
consists of
\begin{itemize}
\item
\emph{the state space}  $\cH$,  a finite-dimensional $\Zt$-graded complex vector space;
\item
\emph{the metric} $\langle \cdot, \cdot \rangle: \cH\otimes \cH\to \C,$ a nondegenerate even
  symmetric pairing;
\item  \emph{the unit}, a distinguished element $\one\in \cH$,
\item
\emph{the operations}, a collection of multilinear $S_r$-equivariant even maps
\begin{equation}
  \label{eq:cohft}
  \Lambda_{g,r}: \cH^{\ot r}\to H^{*}(\oM_{g,r})
\end{equation}
for each $g,r\ge 0$ with $2g+r > 2$
with values in the cohomology of $\oM_{g,r}$.
\end{itemize}

These ingredients are required to be compatible with the natural gluing
and forgetful maps
$$\oM_{g,r+1}\times \oM_{g',r'+1}\to \oM_{g+g',r+r'}, \qquad  \oM_{g,r+2}\to \oM_{g+1,r},
\qquad \oM_{g,r+1}\to \oM_{g,r}.$$

To each CohFT there corresponds a collection of \emph{correlators}, the numerical
invariants given by
\begin{equation}
\label{eq:correlators}
\int_{\oM_{g,r}}\!\!\Lambda_{g,r}(\alpha_1,\ldots,\alpha_r) \prod_{j=1}^r\psi_j^{\ell_j}.
\end{equation}
Here
$\alpha_1,\ldots,\alpha_r\in\cH$
and
$\psi_j=c_1(L_j)$ is the Chern class of the $j$th tautological
line bundle whose fiber is the cotangent line at the $j$th marked point.
Correlators without psi-classes
\begin{equation}
\label{eq:primary}
\langle\alpha_1,\cdots,\alpha_r\rangle_g:=
\int_{\oM_{g,r}}\!\!\Lambda_{g,r}(\alpha_1,\ldots,\alpha_r),
\end{equation}
 (i.e.\ with $\ell_j=0$ for all $j$)  are called
 \emph{primary}.
 Elements $\alpha_j\in\cH$  appearing in  a primary
correlator are called \emph{insertions}.
The \emph{prepotential} of a CohFT $\Lambda$ is the exponential generating function of
its genus-zero primary correlators
\begin{equation}
\label{eq:prepot}
\cF_0(t_1,\ldots,t_n):=
\sum_{r \geq 0} \frac{1}{r!}\left\langle
\bxi\cdot\bt,   \ldots, \bxi\cdot\bt
\right\rangle_{\!0}
=
\sum_{r \geq 0} \sum_{i_1,\ldots,i_r}
\frac{
\langle
\xi_{i_1}, \ldots, \xi_{i_r} \rangle_0}{r!}
t_{i_1} \cdots  t_{i_r},
 \end{equation}
 where $t_1,\dots,t_n$ are formal variables,
 $(\xi_1,\ldots,\xi_n)$ is a basis of the  space $\cH$, and
 $\displaystyle\bxi\cdot\bt=\sum_{j=1}^n\xi_jt_j $.

The prepotential satisfies the so-called
WDVV or the \emph{associativity equation}
which is equivalent to saying that
it equips the state space $\cH$ with a structure of a formal Frobenius manifold (see e.g.~\cite{Dub,Man}).
In particular, genus-zero cohomological degree zero
components $\Lambda^0_{0,r}$ of the CohFT
operations~\eqref{eq:cohft} define on $\cH$ a structure of a two-dimensional topological quantum field
theory (TQFT) or, equivalently, a Frobenius algebra structure with the identity
$\one\in\cH$ and the multiplication
$$
\bullet: \cH\otimes \cH\to \cH
$$
determined by the three-point genus-zero correlators
\begin{equation}
\label{eq:froben}
\langle \alpha\bullet \beta,\gamma  \rangle=\langle \alpha, \beta, \gamma\rangle_0.
\end{equation}
In particular
$$
\langle \alpha, \beta \rangle=\langle \one, \alpha, \beta\rangle_0.
$$

By the Kontsevich-Manin reconstruction theorem~\cite{KM}, the prepotential
$\cF_0$ uniquely determines the genus-zero maps $\Lambda_{0,r}$ of a CohFT.
Moreover, by the results of Givental and Teleman~\cite{Giv, Tel} we know that
if the commutative algebra given by~\eqref{eq:froben} is semisimple then the prepotential
uniquely determines the entire CohFT, i.e.\ the maps $\Lambda_{g,r}$ for all $g$ and $r$.

\subsection{Saito-Givental CohFT (Landau-Ginzburg $B$-model)}

The \emph{Saito-Givental} CohFT (Landau-Ginzburg $B$-models in the physical language)
takes as the input a germ of a holomorphic function
$$\bw:\C^n\to \C$$
with an isolated singularity at the origin.
The state space
of this theory with the corresponding
 algebra structure
is the \emph{Milnor ring} (or the \emph{local algebra}) of the singularity  $\bw$
\begin{equation}
  \label{eq:milnor}
 \cQ_\bw
  :=\C[[x_1,\ldots,x_n]]/\cJ_\bw,
\end{equation}
where
\begin{equation} \label{eq:jac_ideal}
\cJ_\bw:=\langle \partial_1\bw,\ldots,\partial_n\bw \rangle
\end{equation}
is the \emph{Jacobian ideal}
generated by the partial derivatives
$\displaystyle\partial_j\bw=\frac{\partial\bw}{\partial x_j}$ of $\bw$.
The algebra $\cQ_\bw$ is finite-dimensional precisely when $\bw$ has isolated singularity. Its
dimension $\mu_\bw:=\dim \cQ_\bw$ is called the \emph{Milnor number} of the singularity $\bw$.
In what follows, we will refer to the generators
$\partial_j\bw=0$ of  $\cJ_\bw$ as the \emph{Jacobian relations}.

The metric
on the space $\cQ_\bw$ is obtained by identifying it with the space
\begin{equation}\label{Hw-def-eq}
 H(\bw):=\Omega^n(\C^n)/\left( d\bw\wedge\Omega^{n-1}(\C^n)\right)
\end{equation}
 via
 \begin{equation}
   \label{eq:milnor-forms}
   \cQ_\bw \xlongrightarrow{\sim}H(\bw), \quad f\mapsto f\bo,
\text{\ \ where\ \ } \bo=dx_1\wedge \cdots\wedge dx_n,
    \end{equation}
and using the \emph{Grothendieck residue pairing} on $H(\bw)$
\begin{equation}
  \label{eq:groth-res}
\langle f, g \rangle = (f\bo, g\bo):= \Res_\bw(f\!g\bo),
\end{equation}
where
\begin{equation}
\label{residue-formula}
\Res_\bw(f \bo)=\Res_{\C[x]/\C}
\begin{bmatrix}
f(x)\cdot \bo\\
\partial_1\bw, \ldots, \partial_n\bw
\end{bmatrix}.
\end{equation}
The Frobenius manifold structure on $\cQ_\bw$
(i.e.\ the genus-zero part of the CohFT)
constructed by Saito~\cite{Sai,Sai2,Sai3}
depends on a choice of a \emph{primitive} form $\zeta$ on $\cQ_\bw$.
The higher genus maps~\eqref{eq:cohft} $\Lambda^\SG(\bw,\zeta)$ are obtained by applying
Givental's quantization procedure extended in~\cite{Sha,Tel} to the CohFT level.

\subsection{Admissible Landau-Ginzburg pairs and FJRW theory}
\label{sec:fjrw}
While the Saito-Givental CohFT is defined for any isolated singularity $\bw$ (and
requires a choice of a primitive form), $A$-model LG CohFTs are known only for a
quasihomogeneous $\bw$ together with a special group of symmetries.

Recall that a polynomial function $\bw: \C^n\to \C$ is called \emph{quasihomogeneous}
if there exists a collection of positive rational numbers $q_1,q_2,\ldots, q_n$,
called \emph{weights}, such that
for all $\lambda\in \C$ we have
\begin{equation}
  \label{eq:weights}
\bw(\lambda_1^{q_1}x_1,\ldots,\lambda_n^{q_1}x_n)=\lambda\bw(x_1,\ldots,x_n).
\end{equation}
Recall some facts about quasihomogeneous singularities
(see e.g.~\cite[Section 12.3]{AGV}).
\begin{itemize}
\item
  The Milnor number of $\bw$ is given by
\begin{equation}
\label{milnor-number}
\mu_\bw:=\dim\cQ_\bw=\prod_{j=1}^n\left( \frac{1}{q_j}-1\right).
\end{equation}

\item The \emph{socle} of the algebra $\cQ_\bw$ (the subspace of elements of the highest
degree  with respect to the grading induced by the  weights $q_j$) is one-dimensional
and is spanned by the
\emph{socle element}
  \begin{equation}
   \label{eq:hess}
   \mathrm{Hess}(\bw):=
   \det\left( \frac{\partial^2\bw}{\partial x_i\partial x_j} \right)\in \cQ_\bw.
  \end{equation}

\item The degree of $\mathrm{Hess}(\bw)$
(the  \emph{central charge} of the theory) is equal to
  \begin{equation}
    \label{eq:centr-charge}
    \widehat{c}:=n-2\sum_{j=1}^nq_j=\sum_{j=1}^n(1-2q_j).
  \end{equation}

\item We have the following relation (see~\cite[Eq.\ (4.25)]{PV12})
\begin{equation}
\label{hessian-residue}
\Res_\bw (\mathrm{Hess}(\bw)\cdot\bo) =\mu_\bw
\end{equation}
which allows to compute the pairing~\eqref{eq:groth-res} by looking at the highest
degree component of the product $fg$.
\end{itemize}

A quasihomogeneous polynomial is called \emph{nondegenerate} if there is a unique choice
of weights satisfying~\eqref{eq:weights}.
The \emph{group of diagonal symmetries}
\begin{equation}
  \label{eq:diag}
  G_\bw:=\{(\ga_1,\ldots,\ga_n)\in (\C^*)^n \,|\,
  \bw(\ga_1x_1,\ldots,\ga_nx_n)=\bw(x_1,\ldots,x_n)\}
\end{equation}
for a quasihomogeneous $\bw$ is nontrivial and contains
the {\em exponential grading element}
\begin{equation}
\label{exponential-grading}
J_\bw:=  \left(e^{2\pi iq_1}, \ldots, e^{2\pi i q_n}\right).
\end{equation}
If $\bw$ is a nondegenerate polynomial, then the group $G_\bw$ is finite.
A subgroup  $G\subset G_\bw$ is called \emph{admissible} if it contains the
element $J_\bw$.
\begin{df}
A pair  $(\bw, G)$, where $\bw$ is a nondegenerate quasihomogeneous singularity and $G$ is
an admissible subgroup of $G_\bw$ is called an \emph{admissible LG pair}.
\end{df}

In~\cite{FJR07,FJR} Fan, Jarvis and Ruan constructed for every admissible LG pair $(\bw,
G)$ a CohFT with the state space
 \begin{equation}
   \label{eq:fjr}
  \cH(\bw,G):=\bigoplus_{\gb\in G}H(\bw_\gb)^G,
 \end{equation}
 where $\bw_\gb:=\bw|_{V^\gb}$ is the restriction of $\bw$ to the fixed subspace
 $V^\gb:=(\C^n)^\gb$  of $\gb\in G$, and
 the spaces $H(\bw_\gb)$ are defined in~\eqref{Hw-def-eq}.
Notice that, even though $H(\bw_\gb)$ and $\cQ_{\bw_\gb}$ are isomorphic as vector spaces,
the actions of $G$ on them  are not the same and so the invariant subspaces may be
different.
This CohFT is called the FJRW theory since its construction is based on an analysis
of the so-called  Witten equation.

An element  $\gb\in G$ with  a trivial fixed subspace
$V^\gb=0$ is called \emph{narrow}. Elements $\gb\in G$ with $V^\gb\ne0$ are called \emph{broad}.
The summands $H(\bw_\gb)^G  $ of $ \cH(\bw,G)$ are called \emph{sectors} and elements of a sector
corresponding to a broad (resp.~narrow) $\gb\in G$ are called \emph{broad} (resp.~\emph{narrow}).
Correlators~\eqref{eq:primary} $\langle\alpha_1,\cdots,\alpha_r\rangle_g$
with only narrow insertions can be calculated using CohFT formalism and algebro-geometric tools
described in~\cite{FJR07}. However, computations involving broad elements often lead to
insurmountable analytic difficulties and even the three-point correlators defining the FJRW  Frobenius
algebra structure~\eqref{eq:froben} on  $\cH(\bw,G)$ are not known in general.

\subsection{Matrix Factorizations CohFT}
\label{sec:mf-cohft}

In~\cite{PV16} the second and the fourth authors constructed a different $A$-model
Landau-Ginzburg CohFT whose input is also an admissible pair $(\bw, G)$ as in the FJRW theory.
The construction is based on the study of categories of matrix factorizations and is purely
algebraic in nature. We will refer to it as
the \emph{matrix factorization} (MF) CohFT.
Conjecturally, the MF and FJRW CohFTs for the same admissible pair $(\bw,G)$ are equivalent.
Due to technical difficulties of computing FJRW correlators, this
conjecture has been verified only in special cases.
In particular, it is known that the state spaces and the metrics of both theories coincide.
However in the MF CohFT the state space is defined not geometrically,
as a space of Lefschetz thimbles as in the FJRW theory,
but algebraically as the \emph{Hochschild homology of the differential-graded category
  of equivariant matrix factorizations} with the canonical metric.

Let us review the elements of this construction which we will need below.

\subsubsection{The state space and metric of the MF CohFT}

Let $\bw\in\C[x_1,\ldots,x_n]$ be a quasihomogeneous polynomial with an isolated singularity at $0$.
The group $G_\bw$ of diagonal symmetries of $\bw$ (defined by \eqref{eq:diag}) is contained in
the bigger algebraic group $\Ga_\bw$ of diagonal transformations of $\C^n$
preserving $\bw$ up to a scalar. The group $\Ga_\bw$
is equipped with a natural character
\begin{equation}
  \label{eq:char}
  \chi:\Ga_\bw\to \Gm
\end{equation}
 such that $\ker(\chi)=G_\bw$.

Now given any commutative algebraic group $\Ga$ with a homomorphism $\Ga\to \Ga_\bw$
denote by $\chi$ the induced character $\chi:\Ga\to \Gm$ with the kernel $G:=\ker(\chi)$.
Let $\MF_\Ga(\bw)$ be the dg-category
of $\Ga$-equivariant matrix factorizations of $\bw$.
By definition, such a matrix factorization $\ov{E}=(E,\de_E)$ consists of
a $\Zt$-graded $\Ga$-equivariant free $\C[x_1,\ldots,x_n]$-module of finite rank, $E=E_0\oplus
E_1$, together with $\Ga$-equivariant
module maps
$$
\de_1:E_1\to E_0, \de_0:E_0\to E_1\ot \chi, \ \text{ such that } \de_0\de_1=\bw\cdot\id,
\de_1\de_0=\bw\cdot\id.
$$

Assuming that the group $G=\ker( \chi)$ is finite,
the Hochschild homology of the dg-category $\MF_\Ga(\bw)$
has been shown  in~\cite[Eq.\ (2.11)]{PV16} to be
\begin{equation}
\label{hochschild-MF-eq}
\HH_*(\MF_\Ga(\bw))\simeq\bigoplus_{\gb\in G}H(\bw_\gb)^G,
\end{equation}
where the space $H(\bw)$ is given by~\eqref{Hw-def-eq}.
The space $\HH_*(\MF_\Ga(\bw))$ is naturally a module over the dual group $\widehat{G}$,
and the decomposition \eqref{hochschild-MF-eq} is precisely the decomposition into
isotypical components (where elements of $G$ are viewed as characters of $\widehat{G}$).
Furthermore, the map on Hochschild homology induced by the forgetful functor $\MF_\Ga(\bw)\to
\MF(\bw)$ is given by the projection
onto the sector of $\gb=1$ in~\eqref{hochschild-MF-eq}
\begin{equation}
  \label{eq:projection}
\HH_*(\MF_\Ga(\bw))\to H(\bw)^G\subset H(\bw)=\HH_*(\MF(\bw)).
\end{equation}
For each $\Ga$-equivariant matrix factorization $\ov{E}=(E,\de_E)$ of $\bw$, there is a
categorical Chern character $\Ch_G(\ov{E})$ with values in $\HH_*(\MF_\Ga(\bw))$. It is
calculated in~\cite[Eq.\ (3.17)]{PV12}, in terms of the above identification of
the Hochschild homology. In particular, its component in $H(\bw)^G$, which coincides
with the non-equivariant Chern character of $\ov{E}$, is given by
\begin{equation}
\label{super-trace}
\Ch(\ov{E})=\mathrm{str}(\partial_n\de_E\cdots\partial_1\de_E) \cdot \bo \in
\cQ_\bw \cdot \bo=H(\bw)
\end{equation}
(here $\mathrm{str}$ denotes the supertrace of an endomorphism of a $\Zt$-graded vector bundle).

The Hochschild homology $\HH_*(\MF_\Ga(\bw))$
is equipped with a {\em canonical bilinear form}
given by a general categorical construction (see~\cite[Def.\ 2.7.1]{PV16}).
The decomposition \eqref{hochschild-MF-eq} is orthogonal with respect to this form
which was explicitly computed in \cite[Sec.\ 2.7]{PV16}.
Here we will only need the form induced on $H(\bw)^G$ by the projection~\eqref{eq:projection}:
\begin{equation}
\label{mukai-pairing}
\langle f\bo, g\bo \rangle_{\bw}=(-1)^{n\choose 2}\Res_{\bw}(fg \bo),
\end{equation}
where $\Res_{\bw}$ is given by \eqref{residue-formula}.

For an admissible pair $(\bw, G)$, the state space of the \PVs
CohFT coincides with the state space~\eqref{eq:fjr} of the FJRW theory.
Comparing with~\eqref{hochschild-MF-eq}, we see that this space
coincides with the Hochschild homology of the category $\MF_\Ga(\bw)$.
However, in~\cite{PV16} it appears, from the identification
\begin{equation}\label{reduced-HH-eq}
\HH_*(\MF_\Ga(\bw))\ot_R \C \simeq H(\bw)^G,
\end{equation}
as the direct sum of specializations of Hochschild homology spaces
$$\cH(\bw,G)=\bigoplus_{\ga\in G}\HH_*(\MF_{\Ga}(\bw_\ga))\ot_R \C.$$
Here $R=\C[\widehat{G}]$ is the character ring of $G$ acting on $\C$ via the
specialization at $1\in G$ homomorphism $\pi_1:R\to \C$.

Denote by $\cH_\gb:=H(\bw_\gb)^G$ the sector in $\cH(\bw,G)$ corresponding to $\gb\in G$
and by $\lan\cdot,\cdot\ran_{\bw_\gb}$ the pairing~\eqref{mukai-pairing} for the function $\bw_\gb$.

Let
 \begin{equation}\label{zeta-eq}
  \zb=(e^{\pi i q_1},\ldots,e^{\pi i q_n}) \in (\Gm)^n
 \end{equation}
be a special square root  of the exponential grading element $J_\bw \in
G_{\bw}$~\eqref{exponential-grading}.

We equip the state space $\cH(\bw,G)$ with the metric $\lan\cdot,\cdot\ran$ which pairs
the sectors $\cH_\gb$ and $\cH_{\gb^{-1}}$ as follows:
\begin{equation}
\label{PV-pairing}
\lan x_\gb,y_{\gb^{-1}}\ran:= \lan \zb_*x_\gb, y_{\gb^{-1}}\ran_{\bw_\gb},
\end{equation}
where $x_\gb\in \cH_\gb$, $y_{\gb^{-1}}\in \cH_{\gb^{-1}}$.

\subsubsection{Operations of the MF CohFT}
Let us provide some details of the construction of the MF CohFT which will be needed later.
Below we will only consider the case when the group $G$ is the maximal diagonal symmetry
group $G_\bw$ of $\bw$.

The main geometric ingredient of the theory is the collection of moduli spaces $\cSr_g(\ogamma)$
for $g\ge 0$ and $\ogamma=(\gb_1,\ldots,\gb_r)\in G_\bw^r$.
These spaces parametrize rigidified $\Ga_\bw$-spin structures over stacky $r$-pointed
stable curves of genus $g$.
Roughly speaking, such structure is a principal $\Ga_\bw$-bundle $P$ on a curve $\cC$
together with an isomorphism of $\chi_*P$ with the $\Gm$-torsor corresponding to
$\omega^{\log}_\cC$, where $\chi$ is the character~\eqref{eq:char}.
The embedding $\Ga_\bw\subset (\Gm)^n$ associates with a $\Ga_\bw$-spin structure $n$
line bundles $\cL_1,\ldots,\cL_n$ on $\cC$.
For $\bw=x_1^p$ the corresponding line bundle $\cL_1$ is a $p$th root of the bundle
$\omega^{\log}_\cC$. So the notion of a $\Ga_\bw$-spin structure generalizes higher spin
structures~\cite{JKV}.

It is known that the moduli space $\cSr_g(\ogamma)$ is non-empty only when
\begin{equation}
  \label{eq:nonempty}
\gb_1\cdot\ldots\cdot \gb_r=J_\bw^{2g-2+r}.
\end{equation}
For a $\Ga_\bw$-spin curve  $\cC$, consider the map
$$ \rho:\cC\to C $$
to the partial coarse moduli space $C$ obtained by forgetting the stacky structure on
$\cC$ at the marked points. This map gives a line bundle
$$L_j=\rho_*\cL_j$$
on $C$ whose degree is given by the
formula  (see~\cite[Proposition 3.3.1]{PV16} and also~\cite[Proposition 2.2.8]{FJR})
 \begin{equation}\label{deg-L-eq}
 \deg L_j=(2g-2+r)q_j-\th^{(j)}_{\gb_1}-\ldots-\th^{(j)}_{\gb_r},
 \end{equation}
 where rational numbers $\th^{(k)}_{\gb_j}$, $k=1,\ldots,n$,  are given by
 $$ \gb_j=\big(e^{2\pi i\th_{\gb_j}^{(1)}},
 \ldots,e^{2\pi i\th_{\gb_j}^{(n)}}\big) \in (\C^*)^n,  \mathrm{\ with\ } 0\le \th_{\gb_j}^{(k)}<1.$$

This  formula implies the following useful \emph{Selection rule}:
\begin{lem}
If the moduli space $\cSr_g(\ogamma)$ is non-empty, then
\begin{equation}
  \label{select}
  (2g-2+r)q_j-\th^{(j)}_{\gb_1}-\ldots-\th^{(j)}_{\gb_r}\in \Z
\mathrm{\ for \ } j=1,\ldots,n.
 \end{equation}
\end{lem}

The key construction in~\cite{PV16} is that of the \emph{fundamental matrix factorization}
which is a $\Ga_\bw$-equivariant matrix factorization of the polynomial
$\oplus_{i=1}^r\bw_{\gb_i}$  viewed as a function on the space
 $$\cSr_g(\ogamma)\times V^{\gb_1}\times\ldots\times V^{\gb_r}.$$
Using the fundamental matrix factorizations as kernels for functors of the Fourier-Mukai
type and passing to the Hochschild homology, we obtain the maps
\begin{equation}
\label{PV-class}
\phi_{g}(\ogamma):\bigotimes_{i=1}^{r}\cH_{\gb_i}\longrightarrow H^*(\cSr_g(\ogamma), \C)
\end{equation}
Now we can define the operations~\eqref{eq:cohft} of the MF CohFT as%
\footnote{The operations here differ from those   in~\cite[Eq.\ (5.16)]{PV16} by a sign.}
\begin{equation}
\label{reduced-cohft}
\Lambda_{g,r}^\PV(\ogamma)=
{1\over \deg ({\rm st}_g)} \cdot ({\rm st}_g)_*
\phi_{g}(\ogamma):
\bigotimes_{i=1}^{r}\cH_{\gb_i}\longrightarrow H^*(\oM_{g,r},\C).
\end{equation}

The \PVs CohFT~\eqref{reduced-cohft} has many nice properties. First of all, it
satisfies the axioms of the CohFT which connect the restrictions of
$\Lambda^\PV_{g,r}(\ogamma)$ to the boundary divisors with the same maps defined for
other values $(g,r)$. It has a flat identity, which is the natural generator $1_J$ in
$\cH_J=\C$.
In particular, the following {\it metric axiom} holds:
\begin{equation}
\label{metric-axiom}
\lan x_\gb, y_{\gb^{-1}}, 1_J \ran_{0,3}^\PV=\lan (\zb)_*x_\gb, y_{\gb^{-1}}\ran_{\bw_\gb},
\end{equation}
where $x_\gb\in \cH_\gb$, $y_{\gb^{-1}}\in \cH_{\gb^{-1}}$ (see~\cite[Lemma 6.1.1]{PV16}).

Another important property is that, for a polynomial $\bw$ which splits as a disjoint (Thom-Sebastiani) sum,
$$\bw=\bw_1\oplus\bw_2,$$
the corresponding \PVs CohFT decomposes into the tensor product of the CohFTs of the
summands $\bw_1$ and $\bw_2$.
Namely, we have natural identifications
$$G_\bw\simeq G_{\bw_1}\times G_{\bw_2} \mathrm{\ \ and \ \ }
\cH(\bw,G_\bw)\simeq \cH(\bw,G_{\bw_1})\otimes \cH(\bw,G_{\bw_2})$$
under which the map
$\Lambda_{g,r}^\PV(\ogamma_1,\ogamma_2)$ becomes the tensor product of
$\Lambda_{g,r}^\PV(\ogamma_1) $ and $\Lambda_{g,r}^\PV(\ogamma_2)$ (see~\cite[Sec.\
5.8]{PV16}). This property will allow us to focus our attention on polynomials $\bw$ of
one of the three atomic types~\eqref{atomic-type}.

 We refer to~\cite[Sections 5, 6]{PV16} for further properties of
 this CohFT and the corresponding correlators.

\subsection{Invertible polynomials}
\label{sec:lg-mirror}
Recall that a quasihomogeneous nondegenerate polynomial $\bw:\C^n\to \C$ is called
\emph{invertible} if it has $n$ nonzero monomials.%
\footnote{Notice that $n$ is the smallest possible number of monomials for a nondegenerate
  polynomial.}
Invertible polynomials have been classified by Kreuzer and Skarke~\cite[Theorem 1]{KS92}.
They proved that $\bw$ is invertible if and only if it is a disjoint (or Sebastiani-Thom)
sum of polynomials of one of the following three \emph{atomic} types
\begin{equation}\label{atomic-type}
\begin{array}{ll}
  \mathrm{Fermat}: & \bw=x_1^{a_1},
  \\[6pt]
  \mathrm{Chain}: & \displaystyle
\bw=\sum_{i=1}^{n-1}x_i^{a_i}x_{i+1}+x_n^{a_n},
 \\[6pt]
\mathrm{Loop}: & \displaystyle
       \bw=\sum_{i=1}^{n-1}x_i^{a_i}x_{i+1}+x_n^{a_n}x_1,
\end{array}
\end{equation}
where $a_i\geq2$ for each $i\in \{1, 2, \ldots, n\}$.

To study mirror symmetry between the dual
pairs $(\bw, \bw^T)$ of invertible polynomials,
we need to consider symmetry groups.

In Table \ref{table-atomic} we present the exponent matrices and the order of the diagonal symmetry
group $G_\bw$ for each of the atomic polynomials.

\begin{table}[H]
  \centering
\caption{Atomic invertible polynomials}
\label{table-atomic}
\renewcommand{\arraystretch}{1}
 \begin{tabular}{|c|c|c|c|}
 \hline
   & Fermat & Loop& Chain\\
 \hline
$\bw$
&$x^a$
& $\displaystyle\sum_{i=1}^{n-1}x_i^{a_i}x_{i+1}+x_n^{a_n}x_1$
&$\displaystyle\sum_{i=1}^{n-1}x_i^{a_i}x_{i+1}+x_n^{a_n}$\\
\hline
 $E_{\bw}$
&   $\begin{pmatrix}a \end{pmatrix}$  &
$
\begin{pmatrix}
a_1 & 1 &  &  \\
 & a_2 & \ddots &  \\
  & &\ddots & 1 \\
1 & &  & a_n
\end{pmatrix}$
&
$
\begin{pmatrix}
a_1 & 1 &  &  \\
 & a_2 & \ddots &  \\
  & &\ddots & 1 \\
& &  & a_n
\end{pmatrix}$
\\
\hline
$|G_\bw|$&$a$&$\prod\limits_{j=1}\limits^{n}a_j+(-1)^{n+1}$&$\prod\limits_{j=1}\limits^{n}a_j$\\
\hline
\end{tabular}
\end{table}

Let $\bw$ be an invertible polynomial.
Following Kreuzer~\cite{K94}, we will use the entries $\rho_j^{(i)}$ of the inverse of the exponent matrix $E_\bw$:
\begin{equation}\label{exponent-inverse}
E_\bw^{-1} =
\left( \begin{array}{ccc}
\rho_1^{(1)} & \cdots & \rho_n^{(1)} \\
\vdots  & \vdots & \vdots  \\
\rho_1^{(n)} & \cdots & \rho_n^{(n)}
\end{array} \right).
\end{equation}
The sum of the entries in the $i$th row gives the weight of the $i$th variable $x_i$ of $\bw$:
\begin{equation}
\label{variable-weight}
q_i = \sum_{j=1}^n \rho_j^{(i)}.
\end{equation}
The columns of $E^{-1}_\bw$ give special elements in the group of diagonal symmetries of $\bw$
\begin{equation}
  \brho_j:=\left(  e^{2\pi i\rho_j^{(1)}}, \ldots, e^{2\pi i\rho_j^{(n)}}   \right)\in G_\bw.
  \label{rhoj}
\end{equation}
Moreover, their product is equal to the exponential grading element $J_\bw\in G_{\bw}$:
\begin{equation*}
J_{\bw}=\prod_{j=1}^{n}\brho_j.
\end{equation*}

\

We recall some facts  about atomic invertible polynomials (see~\cite{K94, Kr10, HLSW}).

\begin{enumerate}
\item
  For a Fermat polynomial $\bw=x_1^{a_1}$, \ $a_1\geq 2$, we have
  $$q_1=\rho_{1}^{(1)}={1\over a_1}.$$
\item
For a polynomial of the chain type
 $\bw = \sum\limits_{i=1}^{n-1}x_i^{a_i}x_{i+1}+x_n^{a_n}$, the $(i,j)$-th entry of $E_\bw^{-1}$ is
\begin{equation}
\label{chain-entry}\rho_j^{(i)}=
\begin{dcases}
\begin{array}{ll}
(-1)^{j-i}\prod\limits_{k=i}^{j}{1\over a_k}, & j\geq i;\\
0, & j<i.
\end{array}
\end{dcases}
\end{equation}
The weight of the variable $x_i$ is equal to
\begin{equation}\label{chain-weight}
 q_i=\sum_{j=1}^{n}\rho_j^{(i)}=\sum_{j=i}^{n}(-1)^{j-i}\prod\limits_{k=i}^{j}{1 \over  a_k}
 ={1\over a_i}-{1\over a_i a_{i+1}}+ \ldots   +(-1)^{n-i}\prod\limits_{k=i}^{n}{1 \over a_k}>0.
\end{equation}
\item   For a polynomial of the loop type
  $\bw =\sum\limits_{i=1}^{n-1}x_i^{a_i}x_{i+1}+x_n^{a_n}x_1$, we have
\begin{equation}
\label{loop-entry}\rho_j^{(i)}=
\left\{
\begin{array}{ll}
 (-1)^{j-i}\prod\limits_{j+1}^{n}a_k\prod\limits_{k=1}^{i-1}a_k \Big/ \Big(\prod
  \limits_{k=1}^n a_k+(-1)^{n+1}\Big), &  j\geq i,\\
 (-1)^{n+j-i}\prod\limits_{k=j+1}^{i-1}a_k \Big/  \Big(\prod\limits_{k=1}^n a_k+(-1)^{n+1}\Big),  & j<i.
\end{array}
\right.
\end{equation}
Here  we use the convention that an empty product is 1.
The weight of $x_i$ is given by
\begin{equation}\label{loop-weight}
q_i = \left(\sum \limits_{j=i}^{n}(-1)^{j-i}\prod\limits_{k=j+1}^{n}a_k\prod\limits_{k=1}^{i-1}a_k
  +
  \sum\limits_{j=1}^{i-1}(-1)^{n+j-i}\prod\limits_{k=j+1}^{i-1}a_k\right)\bigg/\left(\prod\limits_{k=1}^n
  a_k+(-1)^{n+1}\right).
\end{equation}
\end{enumerate}

Using \eqref{chain-entry}, \eqref{chain-weight}, \eqref{loop-entry}, and
\eqref{loop-weight}, the following result of~\cite{K94} can be obtained.
\begin{prop}
\label{broad-relation}
For each atomic polynomial $\bw$
in Table \ref{table-atomic}, we have:
\begin{enumerate}[(i)]
\item
The rational numbers $\rho_j^{(i)}$ satisfy
\begin{equation}
\label{rho-constraint}
\rho_{j-1}^{(i)}+a_j\rho_j^{(i)}=\de_{j}^{i}.
\end{equation}
\item
Let us set $q_{n+1}:=q_1$ when $\bw$ is a loop and $q_{n+1}:=0$ when $\bw$ is a chain.
Then for all $i=1, \ldots, n$, the rational numbers $q_i$ satisfy
\begin{equation}
\label{weight-dimension}
a_iq_i=1-q_{i+1}.
\end{equation}
\item
The rational number $q_i-\rho_j^{(i)}$ is an integer only if
$j=i+1=n=2$ and $a_n=2.$
\end{enumerate}
\end{prop}

\subsection{Some computational tools}
We finish this section by presenting some facts about the MF CohFT from~\cite{PV16} which we will need later.

\subsubsection{Koszul matrix factorizations and homogeneity conjecture}
Here we recall the definition of Koszul matrix factorizations which appear in some computations below.
We also explain the Homogeneity Conjecture from~\cite{PV16} and present a result
which will be essential for the proof of Proposition~\ref{nonvanishing} in Section~\ref{sec-frob-mfd}.

\begin{df}
Let $V$ be a vector bundle on a scheme $X$ and let $V^\vee$ be the dual bundle. To a
pair of sections $\alpha\in H^0(X,V)$ and $\beta\in H^0(X,V^\vee)$ we associate the
\emph{Koszul matrix factorization} $\{\alpha,\beta\}$ of the function $\bw:=\langle
\beta, \alpha \rangle\in H^0(X,\mathcal{O}_X)$ with the $\Zt$-graded module
$E:=\bigwedge^\bullet(V)$ and the differential $\delta=\alpha\wedge \cdot +
\iota(\beta)$.
\end{df}
When $V$ is a trivial bundle of rank $r$, we  will represent  the sections $\alpha$ and
$\beta$ as $r$-tuples of functions and will write $\{a_1,\ldots,a_r; b_1,\ldots,b_r\}$
instead of $\{\alpha,\beta\}$. Notice that the tensor product of several Koszul matrix
factorizations is also a Koszul matrix factorization.

\begin{df}
Let $\gb=(e^{2\pi i\th^{(1)}},\ldots,e^{2\pi i\th^{(n)}})  \in G_\bw\subset (\C^*)^n$ be
an element of the maximal symmetry group of a quasihomogeneous polynomial $\bw$.

The \emph{degree shifting number} $\iota_\gb$ of $\gb$ is defined by
\begin{equation}
\label{degree-shifting-number}
\iota_\gb:=\sum_{j=1}^{n}(\th_\gb^{(j)}-q_j).
\end{equation}
 For an $r$-tuple $\ogamma=(\gb_1,\ldots,\gb_r)\in G_\bw^r$~,
its \emph{twisted dimension} $\wt{D}_g(\ogamma)$ is defined by
\begin{equation}
\label{twist-dim}
\wt{D}_g(\ogamma) := (g-1)\widehat{c}_\bw + \sum_{i=1}^{r}\iota_{\gb_i} + {1\over 2} \cdot
\sum_{i=1}^{r}n_{\gb_i}~,
\end{equation}
where $\widehat{c}_\bw$ is the central charge~\eqref{eq:centr-charge} and
\begin{equation}
\label{broad-dimension}
n_\gb=\dim V^\gb.
\end{equation}
\end{df}

\

The following {\em Homogeneity Conjecture} is stated in~\cite[Section 5.6]{PV16}.
\begin{conj}
  The image of the map
  $\phi_{g}(\ogamma)$~\eqref{PV-class} is contained in $H^{2\wt{D}_{g}(\ogamma)}(\cSr_g(\ogamma), \C)$.
\end{conj}

\

In~\cite{PV16} a sufficient condition for this conjecture was established.
\begin{lem}
\cite[Corollary 5.6.5]{PV16}
\label{homogeneity-lemma}
The Homogeneity Conjecture holds for the CohFT associated with $\bw$ and $G$ whenever
the space $\HH_*(\MF(\bw_\gb))^G$ is  generated by the Chern characters of Koszul matrix
factorizations for each $\ga\in G$.
\end{lem}

The following result will be reviewed in Section \ref{chern-basis}.
\begin{lem}
\cite[Lemma 2.2]{Gue}
\label{Guere-lemma}
Let $\bw$ be an invertible polynomial. For every $\ga\in G_{\bw}$  the Hochschild
homology $\HH_*(\MF(\bw_\gb))^{G_{\bw}}$ is  generated by the Chern characters of Koszul
matrix factorizations.
\end{lem}
As a consequence of the above two lemmas we have
\begin{prop}
\label{homogeneity-prop}
The Homogeneity Conjecture holds for the CohFT  associated with an invertible polynomial
$\bw$ and the maximal group of diagonal symmetries $G_{\bw}.$
\end{prop}

\subsubsection{Tools for computing three-point correlators}
Here we will derive from~\cite[Proposition 6.2.2]{PV16} a useful result which will help
with computations of genus-zero three-point correlators in Section~\ref{sec:gen-Jac-rel-Fr-alg} below.
First, let us introduce some notation.

Fix $\gb_1,\gb_2,\gb_3\in G$ such that $\cSr_0(\gb_1,\gb_2,\gb_3)$ is non-empty.
In particular, by~\eqref{eq:nonempty}, this means that
$$\gb_1\gb_2\gb_3=J.$$
For each $j=1,\ldots,n$, consider subsets of $\{1,2,3\}$
\begin{equation}
\label{broad-indices}
\Sigma_j:=\{i \ | 1\leq i\leq 3, \gb_i^{(j)}=0\}.
\end{equation}
and for $k=0,1,2$, define the subsets of $\{1,\ldots,n\}$
\begin{equation}
\label{broad-line-indices}
S_k=\{j\ |\ |\Sigma_j|=k \text{ and } L_j\simeq\cO(k-2)\}.
\end{equation}
Finally, for  $\gb,\gb'\in G_\bw$, let
$$
V^{\gb,\gb'}:=V^\gb\cap V^{\gb'}
$$
be the subspace of $V$ fixed by both $\gb$ and $\gb'$.

\begin{prop}
\label{P622-prop}
Assume that $\gb_1,\gb_2,\gb_3$ are such that for every $j$ with $\Sigma_j=\emptyset$, we have $\deg L_j=-1$.

For elements  $t_1,t_2,t_3$ of the torus $(\Gm)^n$, let us
consider the following subspace of
 $V^{\gb_1}\oplus V^{\gb_2}\oplus V^{\gb_3}$:
\begin{align*}
V(t_1,t_2,t_3):=&\{(x_1,x_2,x_3) \in V^{\gb_1}\oplus V^{\gb_2}\oplus V^{\gb_3}\ |\ \pi(x_1,x_2,x_3)=0,
 \ p_{12}(x_1)  = t_1 p_{12}(x_2),  \\  & \ \ p_{23}(x_2)=t_2 p_{23}(x_3), \ \ p_{13}(x_3)=t_3 p_{13}(x_1)\},
\end{align*}
where
\ $p_{ij}:V\to V^{\gb_i,\gb_j}$ \
is the coordinate projection \  and
\ $\pi:V^{\gb_1}\oplus V^{\gb_2}\oplus V^{\gb_3}\to \A^{S_1}$ \ is the composition of
the natural map $V^{\gb_1}\oplus V^{\gb_2}\oplus V^{\gb_3}\to V$ with the projection $p: V \to \A^{S_1}$.

Then there exist elements $t_1,t_2,t_3\in (\Gm)^n$ such that
$$\left( w_{\gb_1}\oplus w_{\gb_2}\oplus w_{\gb_3} \right)|_{V(t_1,t_2,t_3)}=0.$$
Furthermore, the three-point map~\eqref{PV-class}
$$
\phi_0(\gb_1,\gb_2,\gb_3):\cH_{\gb_1}\ot \cH_{\gb_2}\ot \cH_{\gb_3}\to\C$$
is induced on Hochschild homology by the functor
$R\Ga\circ \iota^*$, where
$$\iota: V(t_1,t_2,t_3)\hra V^{\gb_1}\oplus V^{\gb_2}\oplus V^{\gb_3}$$
is the natural embedding.
\end{prop}

\begin{proof}
Note that the first assumption is equivalent to the condition that
$S_0=\emptyset$ and $|\Sigma_j|\ge 1$ for each $j$ such that $\deg L_j=0$.

Now~\cite[Proposition 6.2.2.(i)]{PV16}
implies that the fundamental matrix factorization on
$V^{\gb_1}\oplus V^{\gb_2}\oplus V^{\gb_3}$
is a Koszul matrix factorization $\{\alpha,\beta\}$,
where $\alpha$ and $\beta$ are sections of the dual trivial bundles with bases $(e^*_j)$
and $(e_j)$ numbered by $S_1\sqcup S_2$.
Furthermore, the coefficients $\beta_j$ of $e_j$ in $\beta$ have the following
description. For $j\in S_1$, we have
$$\beta_j=(x_i)_j, \text{ where } \Sigma_j=\{i\}$$
(here we denote by $(x_i)_j$ the coordinates of $x_i\in V^{\gamma_i}\subset \A^n$).
For $j\in S_2$, we have
$$\beta_j=a_j(x_{i_1})_j+b_j(x_{i_2})_j, \text { where } \Sigma_j=\{i_1,i_2\},$$
for some $a_j,b_j\in \C^*$.

Note that the relation $\gb_1\gb_2\gb_3=J_\bw$ implies that $V^{\gb_1}\cap V^{\gb_2}\cap V^{\gb_3}=0$.
Thus, the functions $((x_i)_j)_{\Sigma_j=\{i\}}$ and
$((x_{i_1})_j,(x_{i_2})_j)_{\Sigma_j=\{i_1,i_2\}}$ are exactly the
coordinates on the affine space $V^{\gb_1}\oplus V^{\gb_2}\oplus V^{\gb_3}$.
It follows that the section $\beta$ is regular and its zero locus is the subspace $V(t_1,t_2,t_3)$
for some $t_1,t_2,t_3$.
Thus, the assertion follows from a known property of regular Koszul matrix factorizations
(see~\cite[Proposition 1.6.3.(ii)]{PV16}).
\end{proof}

\begin{crl}\label{P622-cor}
  In the situation of Proposition \ref{P622-prop}, assume
in addition that $V^{\gb_1}=0$ and that the homomorphism
$$G_\bw\to G_{\bw_{\gb_2,\gb_3}}$$
(where $\bw_{\gb_2,\gb_3}=\bw|_{V^{\gb_2,\gb_3}}$) is surjective.
Assume also that for $i=1,2$,
$$\bw_{\gb_i}|_{V^{\gb_i}\cap\ker(p_{S_1})}=p_{23}^*\bw_{\gb_2,\gb_3}|_{V^{\gb_i}\cap \ker(p_{S_1})}.$$
Consider, for $i=2,3$, the linear maps
$$f_i:H(\bw_{\gb_i})\to H(\bw_{\gb_2,\gb_3})$$
induced on the Hochschild homology by the composition of the functors of
the restriction to $V^{\gb_i}\cap \ker(p_{S_1})\subset V^{\gb_i}$
and the push-forward with respect to the projection to $V^{\gb_2,\gb_3}$.
Then for $h_i\in H(\bw_{\gb_i})$, $i=2,3$,
we have
$$\lan \one_{\gb_1},h_2,h_3\ran_{0,3}^{\PV}=\lan f_2(h_2),f_3(h_3)\ran_{\bw_{\gb_2,\gb_3}}.$$
\end{crl}

\begin{proof}
  By Proposition \ref{P622-prop}, there exists an element
$t_2$ of the torus $(\Gm)^n$   such that
$\phi_0(\gb_1,\gb_2,\gb_3)$ is induced by the functor $R\Ga\circ \iota^*$, where
$\iota$ is the embedding of the subspace $V(t_2)\subset V^{\gb_2}\oplus V^{\gb_3}$
consisting of $(x_2,x_3)$ such that $\pi_{S_1}(x_2,x_3)=0$ and $p_{23}(x_2)=t_2p_{23}(x_3)$.
Thus, $V(t_2)$ is the preimage of the graph of $t_2$ on $V^{\gb_2,\gb_3}$ under the surjective map
 $$q:\big(V^{\gb_2}\cap\ker(\pi_{S_1})\big)
\oplus \big(V^{\gb_3}\cap\ker(\pi_{S_1})\big)\to V^{\gb_2,\gb_3}\oplus V^{\gb_2,\gb_3}.$$
By assumption, the restriction of $\bw_{\gb_2}\oplus \bw_{\gb_3}$ to the source of this
map is equal to $q^*(\bw_{\gb_2,\gb_3}\oplus \bw_{\gb_2,\gb_3})$. It follows that
the restriction of $\bw_{\gb_2,\gb_3}\oplus \bw_{\gb_2,\gb_3}$ to the graph of $t_2$ is
zero. Hence, there exists an element $g\in G_{\bw_{\gb_2,\gb_3}}$
such that $t_2=g\zb$, where $\zb$ is the special square root~\eqref{zeta-eq} of the
grading element $J_\bw$.
Since $g$ comes from an element of $G_\bw$, we can replace $t_2$ by $\zb$, which leads to the
claimed formula.
\end{proof}

\section{Mirror Frobenius algebras}
\label{sec-frob-alg}
In this section we establish an isomorphism between two Frobenius algebras related to an
invertible polynomial $\bw$, thus proving LG mirror symmetry at the topological level
(Theorem~\ref{frob-isom}).

In Section~\ref{sec:bases-state-spaces}, we construct bases in the MF state space $\cHw$
and in the Milnor ring $\cQ_{\bw^T}$ of the dual polynomial $\bw^T$.
In Section~\ref{sec:pair-pres-mirror-map}, we
use Kreuzer's work~\cite{K94} to construct a mirror map $\theta$ from $\cQ_{\bw^T}$ to $\cHw$.
In Section \ref{sec:gen-Jac-rel-Fr-alg}, we compute the ring structure constants of $\cHw$
and complete the proof that $\theta$ is an isomorphism of Frobenius algebras.

\begin{rem}
\label{non-atomic}
In what follows we will restrict our attention to atomic polynomials~\eqref{atomic-type}.
This is sufficient, since for a disjoint sum of atomic polynomials
$\bw=\bigoplus\limits_i\bw_i$, the dual polynomial $\bw^T$, the maximal symmetry group
$G_\bw$, and the state spaces $\cH(\bw,G_\bw)$ and  $\cQ_{\bw^T}$ with their metrics decompose accordingly:
$$\bw^T=\bigoplus_i\bw^T_i, \ G_\bw\simeq\prod_iG_{\bw_i}, \
\cH(\bw,G_\bw)\simeq \bigotimes_i\cH(\bw,G_{\bw_i}), \ \mathrm{and} \ \
\cQ_{\bw^T}\simeq\bigotimes_i \cQ_{\bw_i^T}.$$
\end{rem}

\subsection{State spaces}\label{sec:bases-state-spaces}
Let $\bw$ be one of the atomic polynomials~\eqref{atomic-type}
and let $\bw^T$ be its dual~\eqref{eq:dual_poly}.
We will describe bases of the state spaces of the two CohFTs related to $\bw$:
the Milnor ring $\cQ_{\bw^T}$ of the dual polynomial $\bw^T$
and the  state space $\cHw$ for the maximal diagonal
symmetry group $G_\bw$.

\subsubsection{Standard   basis of $\cQ_{\bw^T}$}
\begin{table}[h]
  \centering
\caption{Mirror atomic polynomials}
  \label{table-basis}
\renewcommand{\arraystretch}{1}
 \begin{tabular}{|c|c|c|c|}
 \hline
Type & Fermat & Loop& Chain\\
 \hline
$\bw^T$
&$x^a$
& $\bw^T_{\rm loop}=x_nx_1^{a_1}+\sum\limits_{i=2}^{n}x_{i-1}x_i^{a_i}$
&$\bw^T_{\rm chain}=x_1^{a_1}+\sum\limits_{i=2}^{n}x_{i-1}x_i^{a_i}$\\
\hline
$E_{\bw^T}$
&
$
\begin{pmatrix}
a
\end{pmatrix}$
&$
\begin{pmatrix}
a_1 &  &  & 1 \\
1 & a_2 &  &  \\
  & \ddots &\ddots &  \\
 & & 1 & a_n
\end{pmatrix}$
&
$
\begin{pmatrix}
a_1 &  &  &  \\
1 & a_2 &  &  \\
  & \ddots &\ddots &  \\
 & & 1 & a_n
\end{pmatrix}$
\\
\hline
$\mu_{\bw^T}$
&$a-1$
&$\prod\limits_{j=1}\limits^{n}a_j$
&$\sum\limits_{k=0}\limits^{n}(-1)^{n-k}\prod\limits_{j=1}\limits^{k}a_j$
\\\hline
$\soc(\bw^T)$
&$x^{a-2}$
&$\prod\limits_{i=1}^{n}x_i^{a_i-1}$
&$x_n^{a_n-2}\prod\limits_{i=1}^{n-1}x_i^{a_i-1}$\\
\hline
$x^{\bm}$
&$x^{m}$
&$\prod\limits_{i=1}^{n}x_i^{m_i}$
&$\prod\limits_{i=0}^{k-1}x_{n-2i}^{a_{n-2i}-1}\prod\limits_{j=1}^{n-2k}x_j^{m_j}$
\\
$\fS_{\bw^T}$
 &$0\leq m \leq a-2$ & $0\leq m_i\leq a_i-1$ &
 $0\leq k \leq \lfloor{n\over 2}\rfloor,
$\\
   &&&
$0\leq m_j \leq a_j-1-\de_{j}^{n-2k}$
 \\ \hline
\end{tabular}
\end{table}

In Table~\ref{table-basis} we collect some invariants of the dual polynomial $\bw^T$
including the exponent matrix~\eqref{eq:exp-matrix} $E_{\bw^T}=E^T_{\bw}$,
the Milnor number $\mu_{\bw^T}$
(computed using \eqref{milnor-number} and \eqref{weight-dimension}), and the socle
element~\eqref{eq:hess} $\soc(\bw^T)$.

In the last line of the table we present the basis
\begin{equation}
  \label{eq:basis}
\{x^\bm \, | \, \bm\in\fS_{\bw^T}\}
\end{equation}
 of $\cQ_{\bw^T}$ constructed in~\cite{K94, Kr10}.
This basis consists of monomials $x^\bm:=x_1^{m_1}\ldots x_n^{m_n}$ whose exponent vectors
$\bm=(m_1,\ldots, m_n)$ belong to the set $\fS_{\bw^T}$  of $n$-tuples of non-negative
integers described in Table~\ref{table-basis}.
In particular,  $|\fS_{\bw^T}|=\mu_{\bw^T}$ for any atomic polynomial $\bw^T$.

In the case of the chain polynomial $\bw^T_{\rm chain}$, the set $\fS_{\bw^T_{\rm chain}}$ is partitioned
\begin{equation} \label{chain-decomposition}
  \fS_{\bw^T_{\rm chain}}= \bigsqcup_{k=0}^{\lfloor{n\over 2}\rfloor}
  \fS^k_{\bw^T_{\rm chain}},
\end{equation}
where $\fS^k_{\bw^T_{\rm chain}}$ is the set of the exponents of the monomials \
$ x^{\bm}=\prod\limits_{i=0}^{k-1}x_{n-2i}^{a_{n-2i}-1}\prod\limits_{j=1}^{n-2k}x_j^{m_j}$ \ such
that $0\leq m_j < a_j-\de_{j}^{n-2k}. $
In particular,
\begin{equation}
\label{chain-size}
|\fS^k_{\bw^T_{\rm chain}}|=
\left\{
\begin{array}{ll}
(a_{n-2k}-1)\prod\limits_{j=1}^{n-2k-1}a_j, & k< n/2;\\
1, & k=n/2.
\end{array}\right.
\end{equation}

Let us introduce some terminology related to the bases~\eqref{eq:basis}.
\begin{df}  Let $\bw^T$ be an atomic polynomial.
\begin{itemize}
\item
  The basis $\{x^{\bm}\vert\bm\in\fS_{\bw^T}\}$ is called the
\emph{standard basis} of the Milnor ring $\cQ_{\bw^T}$.
\item
 Elements of $\fS_{\bw^T}$ are called standard vectors.
\item
 The standard vector of the socle element $\soc(\bw^T)$
is called the \emph{socle vector} and
is denoted by
 $\bs(\bw^T)$. It is the maximal element of $\fS_{\bw^T}$ in the lexicographical order.
\item
For a standard vector $\bm \in \fS_{\bw^T}$,
its \emph{complementary vector}
$\ov{\bm}=(\ov{m}_1,
\ldots, \ov{m}_n)$ is
 given by
\begin{equation}\label{dual-vector}
\ov{m}_i=\left\{
\begin{array}{ll}
m_i, &\text{ if }  \bm\in\fS^{k\geq1}_{\bw^T_{\rm chain}} \text{ and }  i> n-2k;\\
s(\bw^T)_i-m_i, & \text{otherwise}.
\end{array}\right.
\end{equation}
\end{itemize}
\end{df}
The standard vectors
from the following example will be important later in the discussion of loop polynomials.
\begin{example}
\label{example-loop}
If $\bw=\sum\limits_{i=1}^{n-1}x_i^{a_i}x_{i+1}+x_n^{a_n}x_1$ is a loop polynomial and $n$ is even.
There are two special standard vectors $\bm^{\rm odd}$ and $\bm^{\rm even}$ in $\fS_{\bw^T}$, which
are complementary to each other, with the components
$$
m^{\rm odd}_i=\left\{\begin{array}{ll}
a_i-1, & i  \text{ is odd};\\
0, & i \text{ is even};
\end{array}\right.
\quad
m^{\rm even}_i=\left\{\begin{array}{ll}
0, & i  \text{ is odd};\\
a_i-1, & i \text{ is even}.
\end{array}\right.
$$
\end{example}

\subsubsection{A basis of $\cHw$ via Chern characters}
\label{chern-basis}
The following result proved by Kreuzer~\cite{K94} (see also~\cite[Section 3.3]{Kr10})
was the first
indication about mirror symmetry for invertible singularities.
\begin{prop}
\label{twisted-rank}
For each invertible polynomial $\bw$,
the dimension of the space $\cHw$ is equal to $\mu_{\bw^T}$.
\end{prop}

According to~\cite[Section 2.4]{Gue}, there exists a basis of $\cHw$ represented by the
Chern characters of Koszul matrix factorizations. To describe this basis explicitly we
introduce a map $\fI$ which assigns to an $n$-tuple $\bm$ of non-negative integers an
element in the symmetry group $G_\bw$ given by
\begin{equation}
\label{group-element}
\fI(\bm):=
\prod_{j=1}^{n}\brho_j^{m_j+1}=J_\bw\prod_{j=1}^{n}\brho_j^{m_j}
\in G_{\bw}.
\end{equation}
We have the following two special values of this map:
\begin{equation}
\label{j-and-inverse}
\fI({\bf 0})=J_\bw  \mathrm{\ \ and\ \  } \fI(\bs(\bw^T))=J_{\bw}^{-1}.
\end{equation}
For chain polynomials the  formula \eqref{chain-entry} implies the following result.
\begin{lem}
\label{chain-broad}
If $\bw=\sum\limits_{i=1}^{n-1}x_i^{a_i}x_{i+1}+x_n^{a_n}$
and $\bm\in\fS^{k}_{\bw^T_{\rm chain}}$, then the element
\begin{equation}
\label{chain-element}
\fI(\bm)=J_{\bw}\prod_{i=0}^{k-1}\brho_{n-2i}^{a_{n-2i}-1}\prod_{j=1}^{n-2k}\brho_j^{m_j}
\end{equation}
fixes the variables $x_n, x_{n-1}, \ldots, x_{n-2k+1}$.
It is narrow if $k=0$ and broad if $k\geq 1$.
\end{lem}

Now we proceed with defining a basis of $\cHw$.
For each narrow $\gb\in G_\bw$, we set $\one_{\gb}$ to be the Chern character of the
trivial matrix factorization of $\bw_\gb=0$:
$$\cH_\gb:=H(\bw_\gb)^{G_{\bw}}\cong\C\{\one_{\gb}\}.$$

To construct the bases of all broad sectors $\cH_\gb$ we need to consider two cases.

\ \\
 \noindent {\sf Case} (i): $\bw=\sum\limits_{i=1}^{n-1}x_i^{a_i}x_{i+1}+x_n^{a_n}x_1$ is a loop polynomial.

\

If $n$ is odd, a direct calculation shows that there
are no broad sectors, since no broad elements can be invariant under the $G_\bw$-action
as required by~\eqref{eq:fjr} (see~\cite[Lemma 1.7]{Kr9}).
So we only need to consider the cases when $n$ is even.
In the notation of Example \ref{example-loop} we have
$$\fI(\bm^{\rm odd})=\fI(\bm^{\rm even})=1\in G_\bw.$$
We consider two Koszul matrix factorizations of $(\A^n, -\bw)$:
\begin{equation}
\label{loop-chern-sign}
\begin{split}
  K_{\rm odd}:=&\bigotimes_{\text{j is even}}\big\{-(x_j^{a_j}+x_{j+1}^{a_{j+1}}x_{j+2}), x_{j+1}\big\},
\mathrm{\ \ and}\\
K_{\rm even}:=&\bigotimes_{\text{j is odd}}\big\{-(x_j^{a_j}+x_{j+1}^{a_{j+1}}x_{j+2}), x_{j+1}\big\}.
\end{split}
\end{equation}
By the supertrace formula \eqref{super-trace}, we have
\begin{equation}\label{chern-odd}
\Ch(K_{\rm odd})=\left(\prod_{j \textit{ is odd}}x_j^{a_j-1}-\prod_{j \textit{ is
      even}}(-a_j x_j^{a_j-1})\right)\bigwedge_{j=1}^{n}dx_j
\end{equation}
and
\begin{equation}
\label{chern-even}
\Ch(K_{\rm even})=\left(\prod_{j \textit{ is odd}}(-a_j x_j^{a_j-1})-\prod_{j \textit{ is
      even}}x_j^{a_j-1}\right)\bigwedge_{j=1}^{n}dx_j\,.
\end{equation}
These Chern characters span the two-dimensional vector space
$\cH_{\gb=1}$.
By \eqref{mukai-pairing} and \eqref{PV-pairing}, we have
\begin{equation}
\label{loop-metric}
\begin{pmatrix}
\langle\Ch(K_{\rm odd}), \Ch(K_{\rm odd})\rangle&\langle\Ch(K_{\rm odd}), \Ch(K_{\rm even})\rangle\\
\langle\Ch(K_{\rm even}), \Ch(K_{\rm odd})\rangle&\langle\Ch(K_{\rm even}), \Ch(K_{\rm even})\rangle\\
\end{pmatrix}
=
\begin{pmatrix}
\prod\limits_{\text{j is even}}(-a_j)
& 1\\
1 &\prod\limits_{\text{j is odd}}(-a_j)
\end{pmatrix}
\end{equation}

\

\ \\
 \noindent {\sf Case} (ii): $\bw=\sum\limits_{i=1}^{n-1}x_i^{a_i}x_{i+1}+x_n^{a_n}$ is a chain polynomial.

\

If $\bm\in\fS^{k\geq1}_{\bw^T_{\rm chain}}$, then $\fI(\bm)\in G_\bw$
fixes
$x_{n-2k+1}, x_{n-2k+2}, \ldots, x_{n}$.
Consider a Koszul matrix factorization  of $(\A^n, -\bw)$
\begin{equation}
\label{chain-koszul}
K_\bm=\left(\bigotimes_{t=n/2-k+1}^{n/2-1}\big\{-(x_{2t-1}^{a_{2t-1}}+x_{2t}^{a_{2t}}x_{2t+1}), x_{2t}\big\}
\right)\ot\big\{(-(x_{n-1}^{a_{n-1}}+x_{n}^{a_n-1}), x_n)\big\}
\end{equation}
Using the supertrace
formula \eqref{super-trace}, the rank one vector space $H_{\fI(\bm)}$ is spanned by
\begin{equation}
\label{chern-chain}
\Ch(K_\bm)=\prod_{\substack{j>n-2k\\ 2\nmid j}} (-a_{j}x_{j}^{a_{j}-1})\bigwedge_{j=n-2k+1}^{n}dx_j\in H_{\fI(\bm)}.
\end{equation}
Again, by \eqref{mukai-pairing} and \eqref{PV-pairing}, we have
\begin{equation}
\label{chain-pairing}
\langle \Ch(K_\bm), \Ch(K_\bm)\rangle=\prod_{j>n-2k, j \textit{ is odd}} (-a_{j}).
\end{equation}

\subsection{A pairing-preserving mirror map}\label{sec:pair-pres-mirror-map}
In~\cite{K94}, a linear map $\theta:\cQ_{\bw^T}\to\cHw$ for each atomic invertible polynomial $\bw$
is constructed. We review this construction here.

Recall that the map $\fI: \N^{n}
 \to G_{\bw}$  is defined in \eqref{group-element}.
If we restrict the map to the set of standard vectors $ \fS_{\bw^T}\subset \N^{n}$, it is almost
one-to-one. The only exception happens when $\bw$ is a loop polynomial
with even number of variables, and in this case, we have
$$\fI(\bm^{\rm odd})=\fI(\bm^{\rm even}).$$
\begin{df}
\label{mirror-map}
The \emph{mirror map} is a linear map
\begin{equation}
  \label{eq:mirror-map}
\th: \cQ_{\bw^T}\longrightarrow \cHw
\end{equation}
defined as follows.
\begin{itemize}
\item If $\fI(\bm)\in G_\bw$ is narrow for $\bm\in\fS_{\bw^T}$, then $\th(x^\bm)=\one_{\fI(\bm)}$;
\item if $\bw=\sum\limits_{i=1}^{n-1}x_i^{a_i}x_{i+1}+x_n^{a_n}x_1$ is a loop polynomial with even
  number of variables, then
$$\th(x^{\bm^{\rm odd}})=\Ch(K_{\rm odd}), \quad \th(x^{\bm^{\rm even}})=\Ch(K_{\rm even});$$
\item if $\bw$ is a chain polynomial and $\bm\in\fS^{k\geq1}_{\bw^T_{\rm chain}}$, then
$$\th(x^{\bm})=\Ch(K_{\bm}).$$
\end{itemize}
\end{df}
We sometimes denote the image $\th(x^\bm)$ of $x^\bm$ by $\th(\bm)$.
We call the vector $\bm$, monomial $x^\bm$, or the value $\th(\bm)$ \emph{narrow}
if the element $\fI(\bm)\in G_\bw$ is narrow, that is, ${\rm Fix}(\fI(\bm))=\{0\}\subset \A^n.$
Otherwise, we call it {\it broad}.
The above discussion can be summarized in the following form.
\begin{lem}
\label{broad-classify}
For any atomic polynomial $\bw$, the element $\th(\bm)\in  \cHw$ is broad
in one of the following cases:
$$\bm=\bm^{\rm odd}, \ \bm=\bm^{\rm even}, \mathrm{\ or\ } \bm\in\fS^{k\geq1}_{\bw^T_{\rm chain}}.$$
\end{lem}

\

Now define the \emph{normalized residue} $\wt\Res_{\bw^T}$ by rescaling  the residue
$\Res_{\bw^T}$~\eqref{residue-formula} so that
\begin{equation}
\label{normalize-residue}
\wt\Res_{\bw^T}({\rm soc}(\bw^T))=1,
\end{equation}
where $\soc(\bw^T)$ is the socle element~\eqref{eq:hess}.
We have a nondegenerate symmetric bilinear pairing on $\cQ_{\bw^T}$,
given by $\wt\Res_{\bw^T}(\cdot ,\cdot )$.
   The nonzero values of this pairing on the basis elements are given by
\begin{equation}
\label{chain-residue}
\begin{array}{lll}
\wt\Res_{\bw^T}(x^{\bm^{\rm even}},x^{\bm^{\rm even}})&=\prod\limits_{j \text{ is odd}}(-a_j),&
\\
\wt\Res_{\bw^T}(x^{\bm^{\rm odd}},x^{\bm^{\rm odd}})&=\prod\limits_{j \text{ is even}}(-a_j),&
\\
\wt\Res_{\bw^T}(x^{\bm}, x^{\ov{\bm}})&=
\prod\limits_{i=0}^{k-1}(-a_{n-2i-1}), &\text{ if } \bm\in\fS^{k\geq1}_{\bw^T_{\rm chain}};\\
\wt\Res_{\bw^T}(x^{\bm}, x^{\ov{\bm}})&=1, & \text{otherwise}.
\end{array}
\end{equation}

By definition~\eqref{PV-pairing} of the A-model pairing $\langle \ , \ \rangle$ on $\cHw$,
we have $\langle \th(\bm), \th(\ov{\bm}) \rangle=1$ if $\bm$ is narrow.
By comparing the A-model
calculations \eqref{loop-metric} and \eqref{chain-pairing} with the normalized residue
calculations in \eqref{normalize-residue} and \eqref{chain-residue},
we obtain the following result.

\begin{prop}
\label{pairing-preserving}
The normalized residue pairing on $\cQ_{\bw^T}$ and the pairing \eqref{PV-pairing} on $\cHw$
correspond to each other via the mirror map \eqref{eq:mirror-map}.
That is, for $\bm, \bm'\in\fS_{\bw^T}$, we have
$$\wt\Res_{\bw^T}(x^\bm, x^{\bm'})=\langle\th(\bm), \th(\bm')\rangle.$$
\end{prop}

\subsection{Generators, Jacobian relations, and Frobenius algebras}\label{sec:gen-Jac-rel-Fr-alg}

In this section we compute the Frobenius algebra structure on the state space $\cHw$ of the \PVs CohFT
and prove Theorem~\ref{frob-isom} (Theorem~\ref{ring-iso} below) verifying Mirror symmetry
at the topological level.

Let $\bm_i$'s be standard vectors in $\fS_{\bw^T}$.
Using~\eqref{eq:froben} we can express the product on $\cHw$ in this basis as follows
\begin{equation}
\label{multiplication-formula}
\th(\bm_1)\bullet\th(\bm_2)
=\sum_{\bm_3, \bm_4\in \fS_{\bw^T}}
\langle\th(\bm_1), \th(\bm_2),\th(\bm_3)\rangle_{0,3}^\PV
\cdot
\eta^{\th(\bm_3), \th(\bm_4)}
\cdot
\th(\bm_4).
\end{equation}
Here $(\eta^{i,j})$ is the inverse of the matrix of the pairing~\eqref{PV-pairing}
in the basis $\{\th(\bm)\}$, where $\bm$'s are the standard vectors in $\fS_{\bw^T}$.

Our computation of the correlator
$\langle\th(\bm_1), \th(\bm_2),\th(\bm_3)\rangle_{0,3}^\PV$ with narrow insertions $\th(\bm_1),
\th(\bm_2),$ and $\th(\bm_3)$ will rely on the Selection rule~\eqref{select},
the Concavity Axiom~\cite[Corollary 5.5.3]{PV16}, and the Index Zero Axiom~\cite[Proposition 5.7.1]{PV16}.

To compute three-point correlators with broad insertions,
we will need more tools and so we start with some preparation.

\subsubsection{Chain type reduction}

Let $\bw(x_1,\ldots,x_n)$ be a nondegenerate quasihomogeneous polynomial of the form
\begin{eqnarray*}
\bw(x_1,\ldots,x_n)&=&x_1^{a_1}x_2+x_2^{a_2}x_3+\ldots+x_{2s-1}^{a_{2s-1}}x_{2s}+x_{2s}^{a_{2s}}x_{2s+1}^m\\
&&+w_0(x_{2s+1},\ldots,x_n),
\end{eqnarray*}
where  $a_i\ge 1$, $a_{2i}\ge 2$, $m\ge 0$, and $2s\le n$
(and $m=0$ if $2s=n$).
Let us consider the Koszul matrix factorization of $w-w_0$,
$$K_0:=\{x_1^{a_1}+x_2^{a_2-1}x_3, x_2\}\ot\ldots\ot \{x_{2s-1}^{a_{2s-1}}+x_{2s}^{a_{2s}-1}x_{2s+1}^m, x_{2s}\}.$$
Note that it has a natural $G_\bw$-equivariant structure.
Set $V=\A^n$, with coordinates $x_1,\ldots,x_n$, $V_0=\A^{[2s+1,n]}$, with coordinates $x_{2s+1},\ldots,x_n$,
and let $p:V\to V_0$ be the natural projection.
For any $\gb\in G_\bw$, consider the functor
$$\Phi_{K_0,\gb}:\MF_{p(G)}(\bw_{0,p(\gb)})\to \MF_{G_\bw}(\bw_\gb): E\mapsto K_0\ot p^*E,$$
where $p:V^\gb\to V^\gb\cap \A^{[2s+1,n]}$ (resp., $(\Gm)^n\to (\Gm)^{[2s+1,n]}$) is the
coordinate projection.
Let $\phi_{K_0,\gb}$ be the induced map on Hochschild homology.

\begin{prop}\label{chain-reduction-prop}
In the above situation, assume that we are given $\gb_1,\gb_2,\gb_3\in G_\bw$, such that
$\cSr_0(\gb_1,\gb_2,\gb_3)$ is non-empty,
$$V^{\gb_1}=0, \ \ V^{\gb_2}\subset \A^{[2s+1,n]},$$
$$S_1=\{2,4,\ldots,2s\}, \ \ V^{\gb_3}=\A^{[1,2s]}\times V^{\gb_2},$$
and $\deg L_j=-1$ for every $j$ with $\Sigma_j=\emptyset$.
Then for any $h_2\in\cH_{\bw_0,p(\gb_2)}=\cH_{\bw,\gb_2}$, $h_3\in \cH_{\bw_0,p(\gb_3)}=\cH_{\bw,\gb_2}$, one has
\begin{equation}
\label{mf-chain-reduction}
\langle 1_{\gb_1}, h_2, \phi_{K_0,\gb_3}(h_3)\rangle_{0,3}^{\PV}=\prod_{i=1}^s(-a_{2i-1})\cdot
\langle h_2, h_3\rangle_{\bw_{\gb_2}}.
\end{equation}
\end{prop}

\begin{proof}
By Corollary \ref{P622-cor}, we have
$$\langle 1_{\gb_1}, h_2, \phi_{K_0,\gb_3}(h_3)\rangle_{0,3}^{\PV}
=\langle h_2,f_3(\phi_{K_0,\gb_3}(h_3))\rangle_{\bw_{\gb_2}},$$
where $f_3:H(\bw_{\gb_3})\to H(\bw_{\gb_2})$ is induced by the restriction to
the subspace $x_2=x_4=\ldots=x_{2s}=0$ followed by the push-forward with respect to the projection
$$p_0:\A^{\{1,3,\ldots,2s-1\}}\times V^{\gb_2}\to V^{\gb_2}.$$
Thus, to calculate $f_3(\phi_{K_0,\gb_3}(h_3))$ we have to calculate the endofunctor of $\MF(\bw_{\gb_2})$,
$$
E\mapsto p_{0*}(\Phi_{K_0,\gb_3}(E)|_{x_2=\ldots=x_{2s}=0}) =
p_{0*}(K_0|_{x_2=\ldots=x_{2s}=0}\boxtimes E) \simeq R\Ga(K_0|_{x_2=\ldots=x_{2s}=0})\ot E.
$$

It remains to observe that
$$K_0|_{x_2=\ldots=x_{2s}=0}=\{x_1^{a_1},0\}\ot\ldots\ot \{x_{2s-1}^{a_{2s-1}},0\}.$$
Hence, in the Grothendieck group of $\MF(\A^{\{1,3,\ldots,2s-1\}},0)$,
this is equal to $(-a_1)\ldots(-a_{2s-1})$ times the class of the stabilization of the origin.
Therefore, we get
$$f_3(\phi_{K_0,\gb_3}(h_3))=\prod_{i=1}^s(-a_{2i-1})\cdot h_3,$$
which implies our formula.
\end{proof}

\subsubsection{Generators}
Using Lemma \ref{broad-classify}, we can classify
   all broad monomials in one variable.
\begin{crl}
\label{broad-element}
Let $\bw$ be one of the atomic invertible polynomials in Table~\eqref{table-atomic}.
If $x_j^{\ell}\in\fS_{\bw^T}$, then $\th(\ell\cdot v_j)$
is broad in exactly one of the following two cases:
\begin{itemize}
\item $\bw$ is a loop polynomial in two variables and $\ell=a_j-1$; \quad or

\item $\bw$ is of the chain type and $(j,\ell)=(n,a_n-1)$.
\end{itemize}

Moreover, in these cases, the element $\fI(v_i)=J_\bw\brho_i\in G_{\bw}$ is broad if
\begin{itemize}
\item $\bw=x^{a_1}_1x_2+x_2^{a_2}x_1$ and $a_i=2$;
\item $\bw=x_1^{a_1}x_2+\ldots+x_{n-1}^{a_{n-1}}x_n+x_n^2$ and $i=n$.
\end{itemize}
\end{crl}

Let $v_j$ be the $n$-tuple of integers whose $j$-th component is $1$ and all other
components are zero. We define
\begin{equation}
\label{j-generator}
\th_j:=\th(v_j).
\end{equation}
We want to show that the elements $\th_1,\ldots,\th_n$ generate $\cHw$ as an algebra.
This will follow immediately from the product formula \eqref{additive} that we will now prove.
 \begin{prop}
 \label{additive-relation}
If both $\bm$ and $\bm+v_j$ are standard vectors in $\fS_{\bw^T}$, then
\begin{equation}
\label{additive}
\th_j\bullet\th(\bm)=\th(\bm+v_j), \quad \ \mathrm{for\ } j=1, \ldots, n.
\end{equation}
\end{prop}
\begin{proof}
  We calculate the product  $\th_j\bullet\th(\bm)$ using the formula \eqref{multiplication-formula}
with   $\th(\bm_1)=\th_j$ and $\th(\bm_2)=\th(\bm)$.
According to the calculation of the pairing $\langle \ , \ \rangle$ via the formula
\eqref{chain-residue} and Proposition \ref{pairing-preserving}, it suffices to compute
both $\langle\th_j, \th(\bm), \bm^{\rm even}\rangle_{0,3}^\PV$ and $\langle\th_j, \th(\bm),
\bm^{\rm odd}\rangle_{0,3}^\PV$ when $\bm+v_j$ is broad,
when $\bw$ is of the loop type,
and $\langle\th_j,  \th(\bm), \th(\ov{\bm+v_j})\rangle_{0,3}^\PV$
in all other cases.
We discuss the details by considering the following three cases, some of which may contain
several subcases.

\

\noindent {\bf Case  1.  Correlators without broad insertions.}

\

Assume first that $v_j, \bm,$ and $\bm+v_j$ are all narrow (and both $\bm$ and $\bm+v_j$ are
standard vectors in $\fS_{\bw^T}$).
We want to calculate the correlator $\langle\theta_j,  \theta(\bm),
\theta(\overline{\bm+v_j})\rangle_{0,3}^{\rm \PV}$.
The formulas \eqref{deg-L-eq} and \eqref{group-element} show that the line bundles $L_i$,
$i=1,\ldots, n$,  satisfy
\begin{equation}
\label{3-point-degree-minus-one}
\deg L_i=-2q_i-\sum_{k=1}^{n} s(\bw^T)_k \rho_{k}^{(i)}=-1.
\end{equation}
Here $s(\bw^T)_k$ is the $k$-th
component of the socle vector $\bs(\bw^T)=(s(\bw^T)_1, \ldots, s(\bw^T)_n)$.

The last equality follows from Proposition~\ref{broad-relation}(i) and
equation~\eqref{variable-weight}.
The equation~\ref{3-point-degree-minus-one} shows that our three-point correlator is concave
and thus by the Concavity Axiom~\cite[Corollary 5.5.3]{PV16} we obtain that
$$\langle\theta_j,  \theta(\bm), \theta(\overline{\bm+v_j})\rangle_{0,3}^{\rm \PV}=1.$$
This equation, together with the formulas for  the multiplication~\eqref{multiplication-formula} and
the paring~\ref{chain-residue}, imply the formula~\eqref{additive} in this case.

\

\noindent {\bf Case 2. Chain correlators with broad insertions.}

\

Now we consider all cases when $\bw$ is a chain polynomial  and at least one of the vectors
 $v_j, \bm, \bm+v_j$ is broad.
Recall that we assume that both $\bm$ and $\bm+v_j$ are standard vectors in $\fS_{\bw^T}$.
Using the description of broad elements in Lemma~\ref{broad-classify}, we will show that
$\bm+v_j$ must be broad.

By Corollary \ref{broad-element}, if $v_j$ is broad, then  we must have $j=n$ and $a_n=2$.
Then Lemma~\ref{broad-classify} implies that the standard vector $\bm+v_j=\bm+v_n$ is broad.
If $v_j$ is narrow, then by Lemma~\ref{broad-classify}
and the description of $\fS_{\bw^T}$ in Table \ref{table-basis}, both $\bm$ and $\bm+v_j$ are broad.
This means that we need to consider the following three subcases:
\begin{enumerate}[(i)]
\item $j\leq n-2k$;
\item $j>n-2k$ and $v_j$ is narrow;
\item $j>n-2k$ and $v_j$ is broad.
\end{enumerate}

By Proposition \ref{pairing-preserving} and the paring calculation~\eqref{chain-residue}, it is
enough to prove that
\begin{equation}
\label{chain-obvious}
\langle\th_j, \th(\bm), \th(\ov{\bm+v_j})\rangle_{0,3}^\PV=\prod\limits_{i=0}^{k-1}(-a_{n-2i-1}).
\end{equation}

The key is to calculate $\deg L_i$ for each $1\leq i\leq n$.
By the definition of $\fI(\bm)$
\eqref{group-element}, we have $\fI(v_j)^{(i)}=q_i+\rho_j^{(i)}$. Using \eqref{chain-entry},  we have
\begin{equation}
\label{sign-q_i}
\begin{array}{ll}
\fI(v_j)^{(i)}>q_i, & \text{ if } j\geq i \text{ and } j-i \text{ is even};\\
\fI(v_j)^{(i)}<q_i, & \text{ if } j\geq i \text{ and } j-i \text{ is odd};\\
\fI(v_j)^{(i)}=q_i, & \text{ if } j<i.\\
\end{array}
\end{equation}

Now we discuss each subcase in detail.
We set $\gb_2=\fI(\bm)$, $\gb_3=\fI(\ov{\bm+v_j})$.

\ \\
     \noindent {\sf Case}
(i):
$j\leq n-2k\leq n-2$.

Then $m_j+1\leq a_j-1$ and $\bm\in\fS^{k}_{\bw^T_{\rm chain}}$ is also broad.
Both $\bm+v_j$ and $\bm$ fix variables $\{x_{n-2k+1}, \ldots, x_n\}$.
Thus if $i>n-2k$, we have
$$\fI(\bm)^{(i)}=\fI(\ov{\bm+v_j})^{(i)}=0.$$
By \eqref{sign-q_i} and the degree formula \eqref{deg-L-eq}, we have
$$\deg L_i=0, \quad \text{if}\quad i>n-2k.$$
If $i\leq n-2k$, the calculation in \eqref{3-point-degree-minus-one} is still valid.
In conclusion, we have
$$
\deg L_i=
\left\{
\begin{array}{ll}
-1, & i\leq n-2k;\\
0, & i> n-2k.
\end{array}
\right.
$$

According to the definitions \eqref{broad-indices} and \eqref{broad-line-indices}, we see that
$$
\Sigma_{n-2k+1}=\ldots=\Sigma_n=\{2, 3\}, \quad S_0=S_1=\emptyset, \quad S_2=\{n-2k+1,
\ldots, n\}.
$$
The assumptions of Proposition \ref{chain-reduction-prop}  are satisfied with $s=0$ and
$$
\bw_{\gb_2}=\bw_{\rm chain}|_{V^{\gb_2}} =\sum\limits_{i=n-2k+1}^{n-1}x_i^{a_i}x_{i+1}+x_n^{a_n}.
$$
Thus, by Proposition \ref{chain-reduction-prop}, Proposition \ref{pairing-preserving}, and formula
\eqref{chain-residue}, we get
\begin{eqnarray*}
\langle\th_j, \th(\bm), \th(\ov{\bm+v_j})\rangle_{0,3}^\PV
&=&\langle \Ch(K_\bm), \Ch(K_{\ov{\bm+v_j}})\rangle_{\bw_{\gb_2}}\\
&=&\prod\limits_{i=0}^{k-1}(-a_{n-2i-1}).
\end{eqnarray*}

\

\ \\
  \noindent {\sf Case}  (ii):  $j>n-2k$ and $v_j$ is narrow.

 \

In this case $n-j$ must be even and $\bm\in \fS^{(n-j)/2}$.
Therefore, $n-j\leq 2k-2$ and so we actually have $j-1>n-2k$.
Thus $\bm$ fixes variables $x_{j-1}, \ldots, x_n$ and $\bm+v_j$ fixes $x_{n-2k+1}, \ldots, x_n$.

Now to calculate $\deg L_i$ we consider the following four cases.
\begin{itemize}
\item
If $i\geq j-1$, then $\deg L_i=0$  since $\fI(\bm)^{(i)}=\fI(\ov{\bm+v_j})^{(i)}=0,$ and $\fI(v_j)^{(i)}=q_i$.
 \item If $n-2k<i<j-1$ and $j-i$ is even, then $\fI(\ov{\bm+v_j})^{(i)}=0$ and we have
  $\fI(v_j)^{(i)}>q_i$ by \eqref{sign-q_i}. By the degree formula \eqref{deg-L-eq}, $\deg L_i=-1.$
 \item  If $n-2k<i<j-1$ and $j-i$ is odd, then  $\fI(\ov{\bm+v_j})^{(i)}=0$ and we have
  $\fI(v_j)^{(i)}<q_i$ by \eqref{sign-q_i}. By the degree formula \eqref{deg-L-eq}, $\deg L_i=0.$
 \item If $i\leq n-2k$, then $\deg L_i=-1$ as in \eqref{3-point-degree-minus-one}.
 \end{itemize}
 In conclusion, we have
 $$
\deg L_i=
\left\{
\begin{array}{ll}
-1, & i\leq n-2k;\\
0, & n-2k<i<j-1 \text{ and } j-i \text{ is odd};\\
-1, & n-2k<i<j-1\text{ and } j-i \text{ is even};\\
0, & i\geq j-1.
\end{array}
\right.
$$

Again by the definitions \eqref{broad-indices} and \eqref{broad-line-indices}, we see that
$$
S_0=\emptyset, \quad S_1=\{n-2k+2, n-2k+4, \ldots, j-2\}, \quad S_2=\{j-1, j, \ldots, n\}.
$$
According to \eqref{chain-koszul}, the Koszul matrix factorization $K_\bm$ and $K_{\ov{\bm+v_j}}$ are given by
$$
K_\bm
=
\left(\bigotimes_{t=j/2}^{n/2-1}\big\{-(x_{2t-1}^{a_{2t-1}}+x_{2t}^{a_{2t}}x_{2t+1}), x_{2t}\big\}
\right)\ot\big\{(-(x_{n-1}^{a_{n-1}}+x_{n}^{a_n-1}), x_n)\big\}
$$
and
$$
K_{\ov{\bm+v_j}}=
\left(\bigotimes_{t=n/2-k+1}^{n/2-1}\big\{-(x_{2t-1}^{a_{2t-1}}+x_{2t}^{a_{2t}}x_{2t+1}), x_{2t}\big\}
\right)\ot\big\{(-(x_{n-1}^{a_{n-1}}+x_{n}^{a_n-1}), x_n)\big\}
$$

Thus, assumptions similar to the ones in
Proposition \ref{chain-reduction-prop} are satisfied with
$$
S_1=\{\ell \big\vert \ell \mathrm{\ is\ even\ },
n-2k<\ell<j-1\}, \ \ V^{\gb_2}=\A^{[j-1,n]}, \ \ V^{\gb_3}=\A^{[n-2k+1,n]}.
$$
Hence, applying Proposition \ref{chain-reduction-prop}, we reduce the calculation to that of the
residue pairing for $\bw_{\gb_2}$. That is,
\begin{eqnarray*}
\langle \th_j, \th(\bm), \th(\ov{\bm+v_j})\rangle_{0,3}^\PV
&=&\prod_{\ell\in S_1}(-a_{\ell-1})\langle \Ch(K_\bm), \Ch(K_\bm)\rangle_{\bw_{\gb_2}}\\
&=&\prod_{\ell\in S_1}(-a_{\ell-1})\prod_{i=0}^{n-j\over 2}(-a_{n-2i-1})\\
&=&\prod\limits_{i=0}^{k-1}(-a_{n-2i-1}).
\end{eqnarray*}
Here the first equality uses \eqref{mf-chain-reduction} and $\Ch(K_{\ov{\bm+v_j}})=\phi_{K_0,\gb_3}(\Ch(K_\bm))$;
the second equality uses \eqref{chain-pairing}.

\

\ \\
\noindent {\sf Case}  (iii):
$j>n-2k$ and  $v_j$ is broad.

In this case, by Corollary \ref{broad-element}, we have $j=n$ and $a_n=2$. We see that
$\fI(v_j)$ fixes variables $x_{n-1}$ and $x_n$ and $\fI(\bm+v_j)$ fixes $x_{n-2k+1}, \ldots, x_n$.
Therefore, we have
$$S_0=\emptyset, \quad S_1=\{\ell \big\vert \ell \mathrm{\ is \ even, \ }
n-2k<\ell<n-1\}, \quad S_{2}=\{n-1, n\}.$$
Similar to Case (ii) above, we obtain
$$\langle \th_j,  \th(\bm), \th(\ov{\bm+v_j})\rangle_{0,3}^\PV=(-a_{n-1})
\prod_{\ell\in S_1}(-a_{\ell-1})=\prod\limits_{i=0}^{k-1}(-a_{n-2i-1}).$$

\

\noindent  {\bf Case 3. Loop correlators with broad insertions.}

\

Finally we consider the cases when $\bw$ is a loop polynomial
and at least one of the elements in $\{v_j, \bm, \bm+v_j\}$ is broad.
Using Lemma \ref{broad-classify} and Corollary \eqref{broad-element}
we see that there are three subcases:
\begin{enumerate}[(i)]
\item $\bm+v_j$ is broad, that is,  $\bm+v_j=\bm^{\rm odd}$ or  $\bm+v_j = \bm^{\rm even}$.
\item $\bm$ is broad, that is,  $\bm=\bm^{\rm odd}$ or $\bm=\bm^{\rm even}$.
\item $v_j$ is broad. This can happen only when $n=2$ and $a_j=2$.
\end{enumerate}

\

\ \\
\noindent {\sf Case}  (i):
We first assume that $\bm+v_j=\bm^{\rm odd}$.

Then $\bm=\bm^{\rm odd}-v_j$ and $j$ should be odd.
We have to compute both $\langle\th_j, \th(\bm), \th(\bm^{\rm even})\rangle_{0,3}^\PV$
and $\langle\th_j, \th(\bm), \th(\bm^{\rm odd})\rangle_{0,3}^\PV.$
For both cases, the first two insertions are narrow and the last insertion is broad.
More explicitly, we have
$$V^{\gb_1}=V^{\gb_2}=0, \ \ V^{\gb_3}=\A^n,$$
$L_{i}=\cO$ if $i$ is even and $L_{i}=\cO(-1)$ if $i$ is odd. Thus,
$$S_1=\{1, 3, \ldots, n-1\}.$$
Hence, by Corollary \ref{P622-cor}, we have
$$f_2(\th(\bm^{\rm odd}-v_j))=\one_{\gb_2},$$
and the correlators are determined by
$f_3(\th(\bm^{\rm even}))$ or $f_3(\th(\bm^{\rm odd}))$ in each case.
Here the map $f_3:H(\bw_{\gb_3})\to H(\bw_{\gb_2})$ is induced by the restriction to
the subspace $x_1=x_3=\ldots=x_{n-1}=0$ followed by the push-forward with respect to the projection
$$p_0:\A^{\{2,4,\ldots, n\}}\times V^{\gb_2}\to V^{\gb_2}.$$
Using the Koszul matrix factorizations defined in~\eqref{loop-chern-sign} and the argument similar
to the one in Proposition \ref{chain-reduction-prop}, we get
\begin{align}
\label{odd-loop}
\langle\th_j, \th(\bm^{\rm odd}-v_j), \th(\bm^{\rm even})\rangle_{0,3}^\PV &=1, \nonumber\\
\langle\th_j, \th(\bm^{\rm odd}-v_j), \th(\bm^{\rm odd})\rangle_{0,3}^\PV  &=\prod_{2\mid k}(-a_{k}).
\end{align}
Using \eqref{multiplication-formula} and \eqref{loop-metric}, the above computation gives
$$\th_j\bullet\th(\bm^{\rm odd}-v_j)=\th(\bm^{\rm odd}).$$

Now assume that $\bm+v_j=\bm^{\rm even}$. Then $\bm=\bm^{\rm even}-v_j$ and $j$ should be even.
Similarly to the previous case we compute
\begin{align*}
 \langle\th_j, \th(\bm^{\rm even}-v_j), \th(\bm^{\rm even})\rangle_{0,3}^\PV&=\prod_{2\nmid k}(-a_{k}), \\
 \langle\th_j, \th(\bm^{\rm even}-v_j), \th(\bm^{\rm odd})\rangle_{0,3}^\PV&=1,
  \end{align*}
which implies that
 $$\th_j\bullet\th(\bm^{\rm even}-v_j)=\th(\bm^{\rm even}).$$

\ \\
    \noindent {\sf Case}  (ii):
If $\bm=\bm^{\rm odd}$, then because $\bm+v_j$ is still a standard vector,
we see that $j$ must be even and $\bm+v_j=\bm^{\rm odd}+v_j$ is narrow.
Then $\ov{\bm+v_j}=\bm^{\rm even}-v_j$.
The above calculation shows that
$$\langle\th_j,  \th(\bm^{\rm odd}), \th(\bm^{\rm even}-v_j)\rangle_{0,3}^\PV=1$$
which implies
$$
\th_j\bullet\th(\bm^{\rm odd})=\theta(\bm^{\rm odd}+v_j).$$

If $\bm=\bm^{\rm even}$ a similar argument gives
$$
\th_j\bullet\th(\bm^{\rm even})=\theta(\bm^{\rm even}+v_j).$$

\ \\
     \noindent {\sf Case}  (iii): Now assume that  $v_j$ is broad.
Without loss of generality, we only need to consider the case when $v_j=\bm^{\rm odd}$. Then $j=1$
and $\bw=x_1^2x_2+x_1x_2^{a_2}$. Using the assumption, $\bm=(0, m)$ for some $m\leq a_2-2$.
We consider the correlator $\langle\th_1, \th(x_2^m),
\th(x_2^{a_1-1-m})\rangle_{0,3}^\PV$. For this correlator,
$$
\Sigma_1=\Sigma_2=\{p_1\}, \quad L_{1}=\cO, \quad L_{2}=\cO(-1), \quad
S_0=S_2=\emptyset, \quad S_1=\{2\}.
$$
This implies
$$\langle\th_1, \th(x_2^m), \th(x_2^{a_1-1-m})\rangle_{0,3}^\PV=1.$$
Thus we obtain $\th_1\bullet\th(x_2^m)=\th(x_2^{m+1})$.

This finishes the proof of Proposition~\ref{additive-relation}.

\end{proof}

\subsubsection{Verifying Jacobian relations for $\cHw$}
 \begin{prop}
 \label{injectiveness}
For each invertible polynomial $\bw$,
the generators $\th_1,\ldots,\th_n$~\eqref{j-generator} of the algebra $\cH(\bw,G_\bw)$ satisfy the
Jacobian relations (i.e.\ generators of the Jacobian ideal~\ref{eq:jac_ideal}) for the Milnor ring
$\cQ_{\bw^T}$~\eqref{eq:milnor} of the mirror polynomial $\bw^T$. In other words, one has
\begin{equation}
\label{jacobi}
 {\partial \bw^T\over \partial x_j}(\th_1, \cdots, \th_n)=0, \quad
\text{\ for\ } j=1,\ldots,n.
\end{equation}
\end{prop}
\begin{proof}
According to the classification of invertible polynomials~\cite{KS92}, it is sufficient to
verify the Jacobian relations~\eqref{jacobi} only for atomic
polynomials~\eqref{atomic-type}. This leads to the following three groups of relations.

\begin{enumerate}[(i)]
\item  For a polynomial of the Fermat type, $\bw=x_1^{a_1}$, we have
\begin{equation}
\label{fermat-relation}
\th_1^{a_1-1}=\th_1\bullet \th_1^{a_1-2}=0.
\end{equation}
\item For
$\bw$ of the chain type \ $\bw=\sum\limits_{i=1}^{n-1}x_i^{a_i}x_{i+1}+x_n^{a_n}$, \ we have
\begin{equation}
\label{chain-last}
a_n\th_{n-1}\bullet \th_n^{a_n-1}=0.
\end{equation}
\item
 For $\bw$ of the loop type, \
 $\bw=\sum\limits_{i=1}^{n-1}x_i^{a_i}x_{i+1}+x_n^{a_n}x_1$, \
 or  of the chain type  with $j\neq n$, we have
\begin{equation}
\label{loop-relation}
a_j\th_{j-1}\bullet \th_{j}^{a_j-1}+\th_{j+1}^{a_{j+1}}=0.
\end{equation}
\end{enumerate}
Here by  $\th_i^k$ we denote the $k$-th power of the generator $\th_i$ with respect to the
multiplication in $\cH(\bw,G_\bw)$.
In the chain case, we use the convention $\th_0:=\one_{J_\bw}$ when $j=1$.
For the loop case, we use the convention $\th_0:=\th_{n}$, $\th_{n+1}:=\th_1$, and $a_{n+1}:=a_1$.

For the relations (i), (ii), and the first case of (iii), we repeat the proof from~\cite{Kr10}.
In the remaining cases, which involve broad insertions, we will have to use the special properties
of the MF CohFT.

\ \\
     \noindent {\sf Case}  (i):
We have $\th_1\bullet \th_1^{a_1-2}=0$, since for $m=0,1,\ldots a_1-2$,
the Selection Rule~\eqref{select} implies
$$\langle\th(1),\th(a_1-2), \th(m)\rangle_{0,3}^\PV = 0.$$

\ \\
     \noindent {\sf Case}  (ii):
 This is the same as~\cite[Lemma 4.6]{Kr10}. To obtain \eqref{chain-last}, it is enough to show that
  $$\langle \th_{n-1},  \th_n^{a_n-1}, \th(\bm)\rangle_{0,3}^\PV=0, \mathrm{\ \ for \ \ }
     \th(\bm)\in \cHw.$$
If it is not true, then using the Selection Rule~\eqref{select} and equations~\eqref{chain-entry}
and \eqref{chain-weight}, we get
$$
\th_{n-1}^{(n)}={1\over a_n}=q_n, \quad \th_{n-1}^{(n-1)}={1\over a_{n-1}}+q_{n-1}\neq q_{n-1},
\quad (\th_n^{a_n-1})^{(n)}=(\th_n^{a_n-1})^{(n-1)}=0.
$$
This implies that
$\fI(\bm)$ fixes the variable $x_n$ but not the variable $x_{n-1}$, which contradicts Lemma~\ref{chain-broad}.

\

\ \\
  \noindent {\sf Case}  (iii):
 According to Corollary  \ref{broad-element}, there are three possible subcases:
\begin{enumerate}
\item   The  monomials
  $\th_{j-1}, \ \th_j^{a_j-1},$ and $\th_{j+1}^k$ for $1\leq k\leq a_{j+1}-1$ are all narrow.
\item $\bw$ is a
  loop polynomial $\bw=x_1^{a_1}x_2+x_1x_2^{a_2}$.
In this case   $\th_i^{a_i-1}$ is broad.
\item $\bw$ is a chain polynomial $\bw=\sum\limits_{i=1}^{n-1}x_i^{a_i}x_{i+1}+x_n^{a_n}$.
In this case $\th_n^{a_n-1}$ is broad.
\end{enumerate}

\ \\
     \noindent {\sf Case} (iii.1):
In this case, using equation~\eqref{additive}, we have
$\th_{j-1}\bullet \th_{j}^{a_j-1}=\th(\bm)$ for a standard vector $\bm$ such that
$\fI(\bm)=J_\bw\brho_{j-1}\brho_j^{a_j-1}$ and
$$m_{i}=q_i+\rho_{j-1}^{(i)}+(a_j-1)\rho_j^{(i)}=q_i-\rho_j^{(i)}+\de_j^{i}.$$
Since none of $\th_{j+1}^k$ for $1\leq k\leq a_{j+1}-1$ is broad,
we see, using Proposition~\ref{broad-relation}.(iii) and Corollary  \ref{broad-element},
that the vector $\bm$ must be narrow. The same is true for
its complementary vector $\ov{\bm}$.

Using \eqref{group-element}, we can
now find $\langle\th_{j+1},\th_{j+1}^{a_{j+1}-1}, \th(\ov{\bm})\rangle_{0,3}^\PV$. We have
\begin{eqnarray*}
\deg L_i&=&q_i-(q_i+\rho_{j+1}^{(i)})-(q_i+(a_{j+1}-1)\rho_{j+1}^{(i)})-(1-q_i-\rho_{j-1}^{(i)}-(a_j-1)\rho_j^{(i)})\\
&=&-1+(\rho_{j-1}^{(i)}+a_j\rho_j^{(i)})-(\rho_{j}^{(i)}+a_{j+1}\rho_{j+1}^{(i)})\\
&=&
\left\{\begin{array}{ll}
0, & j=i;\\
-2, & j=i-1;\\
-1, & \text{ otherwise}.
\end{array}
\right.
\end{eqnarray*}
Here the last equality follows from \eqref{rho-constraint}.
Then by Index Zero Axiom~\cite[Proposition 5.7.1]{PV16},
we obtain
$$\langle\th_{j+1},\th_{j+1}^{a_{j+1}-1},\th(\ov{\bm})\rangle_{0,3}^\PV=-a_j.$$
Now the relation \eqref{loop-relation} follows from
$$\th_{j+1}^{a_{j+1}}=\th_{j+1}\bullet\th_{j+1}^{a_{j+1}-1}
=
\langle\th_{j+1},\th_{j+1}^{a_{j+1}-1},\th(\ov{\bm})\rangle_{0,3}^\PV\ \th(\bm)
=-a_j\th_{j-1}\bullet\th_{j}^{a_j-1}.$$

\ \\
     \noindent {\sf Case} (iii.2):
Now consider the loop polynomial $\bw=x_1^{a_1}x_2+x_1x_2^{a_2}$.
 According to Corollary \ref{broad-element}, the element $\th_i^{a_i-1}=\th(x_i^{a_i-1})$ is broad.
Moreover, $\th_1^{a_1-1}=\Ch(K_{\rm odd})$ and $\th_2^{a_2-1}=\Ch(K_{\rm even})$, where
the Chern characters are given by \eqref{chern-odd} and \eqref{chern-even}.
Take $j=1$ and $\bm=\bm^{\rm even}=(0, a_2-1)$ in \eqref{additive}, we obtain
$$\theta_1\bullet\theta_2^{a_2-1}=\theta\left(x_1x_2^{a_2-1}\right).$$
According to the equation~\eqref{odd-loop}, we have
$$
\langle \th_1, \th_1^{a_1-1}, \th_1^{a_1-2}\rangle_{0,3}^\PV=
\langle \th_1, \th_1^{a_1-2}, \th_1^{a_1-1}\rangle_{0,3}^\PV=-a_2.
$$
Using \eqref{dual-vector}, we have $\ov{(a_1-2, 0)}=(1, a_2-1)$.
Using  \eqref{chain-residue}, we obtain
$$
\th_1\bullet\th_1^{a_1-1}=\langle \th_1, \th_1^{a_1-1}, \th_1^{a_1-2}\rangle_{0,3}^\PV\
\theta\left(x_1x_2^{a_2-1}\right)=-a_2\th_1\bullet\th_2^{a_2-1}.
$$
The other relation $\th_2\bullet\th_2^{a_2-1}=-a_1\th_2\bullet\th_1^{a_1-1}$ in
\eqref{loop-relation} can be obtained similarly.

\ \\
 \noindent {\sf Case} (iii.3): Recall from Table~\ref{table-basis} that the socle element
 of $\bw^T$ is equal to $x_n^{a_n-2}\prod\limits_{i=1}^{n-1}x_i^{a_i-1}$.
To obtain the relation $\th_n\bullet\th_n^{a_n-1}=-a_{n-1}\th_{n-2}\bullet\th_{n-1}^{a_{n-1}-1}$, it
is sufficient to prove the following formula for the three-point correlator
\begin{equation}
\label{jacobi-chain-broad}
\langle \th_n, \th_n^{a_2-1}, \th_n^{a_n-2}\th_{n-2}^{a_{n-2}-2}
\prod_{i=1}^{n-3}\th_{i}^{a_{i}-1}\rangle_{0,3}^\PV=-a_{n-1}.
\end{equation}
We break the discussion into three cases:
\begin{enumerate}[(a)]
\item $n=2$ and $a_2\geq 3$.
\item $n=2$ and $a_2=2$.
\item $n\geq3$ and $a_n=2$.
\end{enumerate}

\ \\
   \noindent {\sf Case} (iii.3a):
We start with $n=2$ and $a_2\geq 3$, i.e., with the chain polynomial
$\bw=x_1^{a_1}x_{2}+x_2^{a_2}.$
We notice that $\th_2^{a_2-1}$ is broad with both variable fixed.
For the correlator
$\langle\th_2,\th_2^{a_2-1}, \th_2^{a_2-2}\rangle_{0,3}^\PV$, we have
 $$\deg L_1=0,\quad \deg L_2=-1, \quad \Sigma_1=\Sigma_2=\{2\}, \quad S_0=S_2=\emptyset, \quad S_1=\{2\}.$$
Now we can apply Proposition \ref{chain-reduction-prop} to obtain \eqref{jacobi-chain-broad}, that is
$$\langle\th_2,\th_2^{a_2-1}, \th_2^{a_2-2}\rangle_{0,3}^\PV=-a_1.$$

\ \\
     \noindent {\sf Case} (iii.3b):
Now we consider the chain polynomial
$\bw=x_1^{a_1}x_{2}+x_2^2$.
Here $a_2=2$ and a direct calculation shows that
\begin{equation}
\label{2-chain-milnor}
\mu_{\bw}=2a_1-1
\end{equation}
and
\begin{equation}
\label{2-chain-hessian}
\mathrm{Hess}(\bw)=-a_1(2a_1-1)x_1^{2a_1-2}\in \cQ_{\bw}.
\end{equation}
The correlator in \eqref{jacobi-chain-broad} is
$\langle\th_2,\th_2^{a_2-1}, \one_J\rangle_{0,3}^\PV$.
Here $x_{2}$ is a broad variable because
$$\th_2^{(1)}=\th_2^{(2)}=0.$$
By \eqref{chern-chain}, we have
$$\th_2=a_{1}x_{1}^{a_{1}-1}dx_1\wedge dx_2.$$
We obtain
\begin{eqnarray*}\langle \th_2,  \th_2,  \one_J\rangle_{0,3}^\PV
&=&\langle -a_{1}x_{1}^{a_{1}-1}dx_1\wedge dx_2, a_{1}x_{1}^{a_{1}-1}dx_1\wedge dx_2\rangle_{\bw}\\
&=&\Res_\bw\left((a_1x_1^{a_1-1})^2 dx_1\wedge dx_2\right)\\
&=&{a_1\over 1-2a_1}\mu_{\bw}\\
&=&-a_1.
\end{eqnarray*}
Here the first identity is the metric axiom \eqref{metric-axiom}; the second identity
follows from the formula for the paring~\eqref{mukai-pairing};
the third identity uses the calculation \eqref{2-chain-hessian} and the residue formula
\eqref{hessian-residue}; the last identity follows from the equation~\eqref{2-chain-milnor}.
This proves the equation~\eqref{jacobi-chain-broad}.

\ \\
     \noindent {\sf Case} (iii.3c):
Finally, we compute
$\langle \th_n,  \th_n, \th_{n-2}^{a_{n-2}-2}\prod_{i=1}^{n-3} \th_{i}^{a_{i}-1}\rangle_{0,3}^\PV$
when $n\geq3$ and $a_n=2$.
Using $\th_n^{(i)}=q_i+\rho_n^{(i)}$, \eqref{chain-entry} and \eqref{chain-weight},
we get $\th_n^{(n)}=1$, $\th_n^{(0)}=0$, for all $i\leq n-2$, $q_i< 2\th_n^{(i)}< 1+q_i.$
Thus we obtain
$L_{n}=0, L_{n-1}=0, \text{ and } L_i= -1 \text{ if } i\leq n-2.$
We have
$$
\Sigma_{n-1}=\Sigma_n=\{1, 2\}, \quad S_0=S_1=\emptyset, \quad S_2=\{n-1,n\}, \quad
V^{\gb_2}=V^{\gb_3}=\A^{[n-1,n]}.
$$
Now again, we apply Proposition \ref{chain-reduction-prop} to obtain \eqref{jacobi-chain-broad}.
This completes the proof of Proposition~\ref{injectiveness}
\end{proof}
The following immediate implication of relations~\eqref{chain-last}
and~\eqref{loop-relation} will be used later.
\begin{crl}
\label{vanishing}
For a polynomial  of the chain type, $\bw=\sum\limits_{i=1}^{n-1}x_i^{a_i}x_{i+1}+x_n^{a_n}$, we have
\begin{equation}
  \label{eq:vanish}
\th_{i-1}\bullet \th_i^{a_i}=0, \  \text{\ for \ } i=2,\ldots,n.
\end{equation}
\end{crl}

\subsubsection{Mirror symmetry between Frobenius algebras}
Now we are ready to prove mirror symmetry at the Frobenius algebra level.
\begin{thm}
\label{ring-iso}
Let $\bw$ be an invertible polynomial. The mirror map $\th$ defined in Definition
\ref{mirror-map} is an isomorphism of  Frobenius algebras
$$
\th: \Big(\cQ_{\bw^T}, \wt\Res_{\bw^T}(,), \cdot\Big) \longrightarrow
  \Big(\cHw, \langle , \rangle, \bullet \Big).
$$
\end{thm}
\begin{proof}
Since the Frobenius algebras on both sides are isomorphic to tensor products
of the algebras of the corresponding atomic components (for the $A$-side this follows
from~\cite[Theorem 5.8.1]{PV16}), the statement reduces to the case when $\bw$ is atomic.

When $\bw$ is atomic, Proposition~\ref{additive-relation}
shows that the elements $\th_1,\th_2,\ldots,\th_n$ generate the algebra $\cHw$;
Proposition~\ref{injectiveness} gives that the mirror map $\th$
is an algebra homomorphism; and Proposition~\ref{pairing-preserving}
establishes that the pairings agree under the isomorphism $\th$.
\end{proof}

\section{Mirror Frobenius manifolds}
\label{sec-frob-mfd}
In this section, we prove Theorem~\ref{main-theorem}
establishing mirror symmetry in the genus-zero case
(i.e.\ isomorphism of Frobenius manifolds).

By the definition of the prepotential~\eqref{eq:prepot},
once we fix an isomorphism between the spaces $\cQ_{\bw^T}$ and $\cHw$,
all we need is to identify the corresponding correlators.
After choosing the standard basis of $\cQ_{\bw^T}$ and the isomorphism $\th$ from Theorem
\ref{ring-iso}, we will do this in three steps. First we prove a nonvanishing result
Proposition~\ref{nonvanishing}. Then, we use it and the WDVV equations to show, in
Proposition~\ref{reconstruction}, that all genus-zero correlators can be reconstructed from
the Frobenius algebra structure constants (three-point correlators) and several genus-zero
four-point correlators.  Finally, in Section~\ref{reconstruction},
we compute these correlators and match them with the corresponding correlators for the
Saito Frobenius manifold.

\subsection{Nonvanishing}
Recall that the ring $\cHw$ is generated by $\th_1, \ldots, \th_n$, so any genus-zero
$k$-point \PVs correlator can be written
as a linear combination of correlators
\begin{equation}
\label{genus0-pv}
\left\langle \prod_{i=1}^n \th_i^{e_{1,i}}, \ldots, \prod_{i=1}^n \th_i^{e_{k,i}}
\right\rangle^\PV_{0,k} \quad
\mathrm{with} \quad
e_{j,i}\in \mathbb{Z}_{\geq 0}.
\end{equation}
Sometimes we omit the superscript and subscript when the notation does not cause confusion.
From Theorem~\ref{ring-iso}, we see that
each nonzero element $\prod_{i=1}^n \th_i^{e_{j,i}}$ must belong to $\cH_{\gb_j}$ for some $\ga_j\in G_\bw$.
We label the decorations of this correlator by $\ogamma=(\ga_1,\ldots,\ga_k)$ and denote
by $\cS_0(\ogamma)$ the corresponding component of the moduli space of $\Ga_{\bw}$-spin structures.
Let ${\rm st}: \cS_{0}(\ogamma)\to \oM_{0,k}$ be the morphism forgetting the $\Ga_{\bw}$-spin structure.
For each nonzero element $v_i\in \cH_{\gb_i}$, using notation from
\eqref{degree-shifting-number} and \eqref{broad-dimension}, we define
\begin{equation}
\label{insertion-degree}
\deg v_i:={n_{\gb_i}\over 2}+\iota_{\gb_i}.
\end{equation}

We now give conditions for nonvanishing of the correlator \eqref{genus0-pv}.
\begin{prop}[Nonvanishing]
\label{nonvanishing}
If the  genus-zero $k$-point correlator \eqref{genus0-pv} is nonzero, then
\begin{equation}\label{integer-deg}
-2q_j - \sum_{\nu=1}^k \sum_{i=1}^n\rho_i^{(j)}e_{\nu,i}  \in \mathbb{Z},  \quad \text{for}\;\;j=1, \ldots, n,
\end{equation}
and
\begin{equation}
\label{homogeneity-eq}
\sum_{\nu=1}^k \deg\Bigl( \prod_{i=1}^{n} \th_i^{e_{\nu,i}}\Bigr)=\widehat{c}_{\bw}+k-3.
\end{equation}
\end{prop}
\begin{proof}
If the correlator~\eqref{genus0-pv} is nonzero, then the moduli space $\cS_0(\ogamma)$ is non-empty.
According to Selection rule~\eqref{select}, the degrees of
$L_j=\rho_*\cL_j$, $i=1,\ldots,n,$
must be integers. Using \eqref{deg-L-eq}, we obtain the equation \eqref{integer-deg}.

According to Proposition \ref{homogeneity-prop}, the Homogeneity Conjecture holds for
MF CohFT associated with $(\bw, G_{\bw})$. This implies that for the correlator \eqref{genus0-pv},
the image of the map $\phi_0(\ogamma)$~\eqref{PV-class} in $H^*(\cS_0(\ogamma))$ has
a pure cohomological degree $2\wt{D}_{0}(\ogamma)$.
According to \eqref{twist-dim} and \eqref{insertion-degree}, we have
$$
2\wt{D}_{0}(\ogamma)=-2\widehat{c}_{\bw}+\sum_{i=1}^{k}(n_{\gb_i}+2\iota_{\gb_i})=
-2\widehat{c}_{\bw}+\sum_{\nu=1}^k 2\deg\prod_{i=1}^{n} \th_i^{e_{\nu,i}}.
$$
Since the space $\cS_0(\ogamma)$ has real dimension $2k-6$,
if the correlator is nonzero, we must have $2\wt{D}_{0}(\ogamma)=2k-6$.
Now this is equivalent to \eqref{homogeneity-eq}.
\end{proof}
\begin{rem}
The fact that $\widehat{c}_\bw=\widehat{c}_{\bw^T}$ implies that under the mirror
map~\eqref{mirror-map} the nonvanishing conditions here are the same as the nonvanishing conditions in~\cite[Lemma 4.1]{HLSW}.
\end{rem}

\subsection{Reconstruction}
The nonvanishing condition from Proposition~\ref{nonvanishing} and the WDVV
(associtiativity) equations allow us to reconstruct all genus-zero primary \PVs
correlators from the Frobenius algebra structure on $\cHw$ given in Theorem \ref{ring-iso}
and a few genus-zero four-point correlators.
The proof of this is done along the same steps as the proof of reconstruction in~\cite[Sections 5 and 7]{HLSW}, replacing FJRW
invariants with \PVs invariants.
We will not reproduce all details of the reconstruction process from~\cite{HLSW} here.
Instead, for reader's convenience, we will give several examples demonstrating the use
of the nonvanishing conditions and the WDVV equations.
 In the Appendix, we provide a slightly simpler proof of the reconstruction for polynomials
 of the chain type, which also deals with several cases not considered in~\cite{HLSW}.

The following lemma, which is a consequence of the WDVV equations, is proved exactly as
its counterpart~\cite[Lemma 6.2.6]{FJR} in the FJRW theory.
\begin{lem}[WDVV reduction]\label{WDVV}
Genus-zero $k$-point \PVs correlators satisfy
\begin{align}
  \langle \xi_1, \ldots, \xi_{k-3}, \ga, \de, \epsilon \bullet \phi \rangle
  = &\langle \xi_1, \ldots, \xi_{k-3}, \ga,  \epsilon, \de \bullet \phi \rangle\nonumber
+ \langle \xi_1, \ldots,\xi_{k-3}, \ga \bullet \epsilon, \de, \phi \rangle \nonumber \\
-&\langle \xi_1,\ldots,\xi_{k-3}, \ga \bullet \de, \epsilon, \phi \rangle
   +S \label{reconst_eq},
\end{align}
where $S$ is a linear combination of correlators with fewer than $k$ insertions.
If $k=4$, then
$S=0$, i.e.\ there are no such terms in the equation.
\end{lem}
This lemma is a key tool in the reconstruction process.
By appropriately choosing the insertions in~\eqref{reconst_eq}, we can ensure
that the correlators in the right-hand side are simpler than the one in the left-hand side.

The following important illustration of this idea is the first step of the
reconstruction.
\begin{lem}\label{Step 1}
For $k\geq 4$,
any genus-zero $k$-point \PVs correlator $\langle \xi_1, \ldots, \xi_{k}\rangle$
can be represented
as a linear combination of special
correlators of the form
 \begin{equation}\label{C2}
X = \langle \underbrace{\th_n, \ldots, \th_n}_{\ell_n}, \underbrace{\th_{n-1}, \ldots,
  \th_{n-1}}_{\ell_{n-1}}, \ldots, \underbrace{\th_1, \ldots, \th_1}_{\ell_1}, \alpha, \beta\rangle_{0, r},
\end{equation}
where $r\leq k$, \ $\th_1,\ldots,\th_n$ are standard generators~\eqref{j-generator} of the
algebra $\cHw$, and the elements $\alpha, \beta \in \cHw$ are products of some $\th_i$'s.
\end{lem}
\begin{proof}

Since $\th_1,\ldots,\th_n$ generate the algebra $\cHw$, we may assume that each insertion
$\xi_j$ is a product of $\th_i$'s. We call the number of factors in such a product the
\emph{degree} of the corresponding insertion. Using the string equations,\footnote{The string
  equations are consequence of the existence of a flat unit in a CohFT,   see~\cite[Theorem 5.1.3]{PV16}.} we can further
assume that there are no identity elements among $\xi_1, \ldots, \xi_k$.

If there are at least three insertions, say $\xi_1, \xi_2$ and $\xi_3$, of degree greater than $1$
(i.e.\ they are products of at least two $\th_i$'s),
we factor $\xi_3$ as $\xi_3=\th_j\bullet \wt{\xi_3}$, for some $j$,
and apply~\eqref{reconst_eq} with $\epsilon=\th_j$, $\phi=\wt{\xi_3}$, $\ga=\xi_1$ and
$\de=\xi_2$.

Observe now that in each of the three correlators
in the right-hand side of~\eqref{reconst_eq}, at least one of the last three insertions
is $\th_j$ or $\wt{\xi_3}$,
whose degree is smaller than the degree of $\xi_3$.
If we start with three insertions of the largest possible degrees, we see
that this process allows to decrease the sum of the smallest $k-2$ degrees
of insertions until they all become equal to $1$. This gives the result.
\end{proof}

Next, we will show that for atomic polynomials from Table~\ref{table-basis}, any special
correlator~\eqref{C2} can be reconstructed from correlators of a simpler type.
For this we will rewrite the nonvanishing conditions of Proposition~\ref{nonvanishing} for
correlator~\eqref{C2} in a different way.
Write elements $\alpha, \beta\in \cHw$ as products of generators:
\begin{equation}\label{basis-assumption}
  \alpha=\prod_{i=1}^{n} \th_i^{P_i}, \quad \beta=\prod_{i=1}^{n} \th_i^{Q_i}
\end{equation}
with $(P_1, \ldots, P_n), (Q_1, \ldots, Q_n)
\in \fS_{\bw^T}.$
From the description of the set $\fS_{\bw^T}$ in the last row of Table~\ref{table-basis},
we see that for nonvanishing $\alpha$ and $\beta$ we have
\begin{equation}
\label{expon-limits}
P_i\leq a_i-1 \mathrm{\ and\ } Q_i\leq a_i-1, \mathrm{\ for\ } i=1,\ldots,n.
\end{equation}

Now we define rational numbers $b_i$ and $K_i$, for $i=1,\ldots, n$, by
\begin{equation}\label{b_def}
\left(\begin{array}{c}
b_1\\
\vdots\\
b_n \end{array}\right)= E_{\bw}^{-1} \left(\begin{array}{c}
\ell_1 + P_1 + Q_1 + 2\\
\vdots\\
\ell_n + P_n + Q_n + 2 \end{array}\right),
\end{equation}
and
\begin{equation}\label{ind-k}
K_i= \ell_i - b_i + 1.
\end{equation}
The nonvanishing conditions \eqref{integer-deg} and \eqref{homogeneity-eq} impose strong
constraints on these numbers.
\begin{lem}
\label{lem:sumK}
For a correlator $X$ given by \eqref{C2}, the numbers $K_i$ are integers satisfying
\begin{equation}\label{sumK}
  K_1+K_2+\ldots+K_n =1.
\end{equation}
\end{lem}
\begin{proof}
Applying \eqref{variable-weight} and \eqref{integer-deg} to the correlator $X$,
  we see that   $b_i\in \Z$, thus $K_i\in \Z$.
 Now the equation \eqref{sumK} follows from \eqref{homogeneity-eq}.
\end{proof}
For one variable Fermat polynomials, this gives us the correlators of a very simple form.
\begin{example}
In the Fermat case, $\bw=x_1^a$, equation~\eqref{sumK} gives $K_1=1$,
and~\eqref{b_def} and \eqref{ind-k} imply
$$b_1={\ell_1+P_1+Q_1+2\over a}=\ell_1.$$
Since $P_1, Q_1\leq a-2$, we obtain $\ell_1\leq 2$.
If $\ell_1=2$, then $P_1+Q_1+4=2a$ and we must have $P_1=Q_1=a-2$.
Now we see that
a nonvanishing correlator~\eqref{C2} must be either a three-point correlator or the
four-point correlator $\langle\th_1, \th_1, \th_1^{a-2}, \th_1^{a-2}\rangle$.
\end{example}
The reconstruction
processes for polynomials of chain and loop types
are much more complicated. We state the result below in Proposition \ref{reconstruction}.
For polynomials of the loop type, the same proof as the proof of the similar result in the
FJRW theory~\cite[Theorem 5.19]{HLSW} works after replacing FJRW correlators by \PVs correlators.
In the Appendix, we present our proof for polynomials of the chain type. It is slightly
shorter than the corresponding proof for the FJRW theory given in~\cite{HLSW} and also
treats some cases not considered there.

\begin{prop}
[Reconstruction]\label{reconstruction}
The Frobenius manifold structure of the \PVs CohFT for an atomic polynomial $\bw$ is
uniquely determined by its Frobenius algebra and the following genus-zero four-point
correlators
$$
\begin{array}{ll}
  \fF_1= \langle \th_1, \th_1, \th_1^{a-2}, \th_{J^{-1}}\rangle_{0,4}^\PV,
  & \text{\ for \ } \bw=x_1^a \text{ of the Fermat type};\\
  \fF_i=\langle \th_i, \th_i, \th_{i-1}\th_i^{a_i-2}, \th_{J^{-1}}\rangle_{0,4}^\PV,
 & \text{\ for \ } \bw \text{ of a loop type, \ }   i=1,\ldots,n;  \\
  \fF_n=\langle \th_n, \th_n, \th_{n-1}\th_n^{a_n-2}, \th_{J^{-1}}\rangle_{0,4}^\PV,
   & \text{\ for \ } \bw \text{ of a chain type}.
\end{array}
$$

Here $\th_{J^{-1}}:=\th(\soc(\bw^T))$,
and we use the convention $\th_{0}:=\th_n$.
\end{prop}

Note that $\th_{J^{-1}}=\one_{J^{-1}}\in \cH_{J^{-1}}$ by \eqref{j-and-inverse}.
For the correlators   $\fF_i$ from Proposition~\ref{reconstruction}, the numbers
$b_j$~\eqref{b_def} and $K_j$~\eqref{ind-k} are given by
$$(b_1, \ldots, b_{n-1}, b_n)=(1, \ldots, 1, 2)
\text{\ \ and \ \ }
(K_1, \ldots, K_{n-1}, K_n)=(0, \ldots, 0, 1).$$

\subsection{Calculations of genus-zero four-points \PVs correlators}
To prove Theorem \ref{main-theorem}, it remains to compute the correlators $\fF_i$ from
Proposition~\ref{reconstruction} and match them with the corresponding correlators from
Saito's Frobenius manifold~\cite{Sai, LLS}. The latter have been calculated
in~\cite{HLSW}, using the perturbative formula for primitive forms developed
in~\cite{LLS}. Using Proposition \ref{reconstruction} and~\cite[Proposition 6.8]{HLSW},
Theorem \ref{main-theorem} will follow from
\begin{prop}
\label{mirror-invariant-prop}
For every correlator $\fF_i$ in Proposition \ref{reconstruction}, we have
\begin{equation}
\label{mirror-invariant}
\fF_i=-q_i,
\end{equation}
where $q_i$ is the weight of the variable $x_i$ in the weighted homogeneous polynomial $\bw$.
\end{prop}
The rest of the paper will be devoted to the proof of Proposition~\ref{mirror-invariant-prop}.

\subsubsection{Concavity}
We separate the correlators $\fF_i$ in
Proposition \ref{reconstruction} into two cases:
\begin{enumerate}[(i)]
\item The correlator is concave.
\item The correlator is not concave.
\end{enumerate}
According to the classification in~\cite[Lemma 6.5 and Lemma 6.6]{HLSW},
modulo a cyclic permutation of the variables in a polynomial of the loop type, the only
non-concave correlators in Proposition~\ref{reconstruction} is
$$\fF_n =\langle \th_n, \th_n, \th_{n-1}\th_n^{a_n-2}, \th_{J^{-1}}\rangle_{0,4}^{\PV},$$
where the polynomial $\bw$ belongs to one of the four cases listed in Table~\ref{nonconcave}.

\begin{table}[H]
\centering
\caption{Non-concave correlators $\fF_n$}
\label{nonconcave}
\renewcommand{\arraystretch}{2}
 \begin{tabular}{|c|c|c|c|c|c|}
 \hline
Type
& Polynomial $\bw$
& Constraints
&$\th_{n}$
&$\th_{n-1}\th_{n}^{a_n-2}$
&$\th_{J^{-1}}$
\\
\hline
(a)
&$\sum\limits_{i=1}^{n-1}x_i^{a_i}x_{i+1}+x_n^2x_1$
&$ n\geq3$
& narrow
& narrow
&narrow
\\
\hline
(b)
&$\sum\limits_{i=1}^{n-1}x_i^{a_i}x_{i+1}+x_n^2x_1$
&$n=2, a_1\geq 3$
&broad
&narrow
&narrow
\\
\hline
(c)
&$\sum\limits_{i=1}^{n-1}x_i^{a_i}x_{i+1}+x_n^2x_1$
& $n=2$, $a_1=2$
&broad
&broad
&narrow
\\
\hline
(d)
& $\sum\limits_{i=1}^{n-1}x_i^{a_i}x_{i+1}+x_n^2$
&none
&broad
&narrow
&narrow
\\\hline
\end{tabular}
\end{table}

More explicitly, in case (a), all the insertions are narrow but the correlator is not concave.
In the remaining three cases, there exists at least two broad insertions in the correlator.

The calculations for concave
correlators are exactly the same as the calculations of the
corresponding FJRW invariants, which was done
in~\cite[Section 6]{HLSW}.
For nonconcave correlators, we will first use Gu\'er\'e's formula~\cite{Gue} to calculate
the virtual class of the \PVs CohFT.
Indeed, Gu\'er\'e's formula can be applied to all the cases in
Proposition~\ref{reconstruction}.
In the concave case, it is just the Witten's top Chern class defined in~\cite{PV01}.

To describe and use Gu\'er\'e's formula, we need some combinatorial preparation.
Consider the correlator
$\fF_n =\langle \th_n, \th_n, \th_{n-1}\th_n^{a_n-2}, \th_{J^{-1}}\rangle_{0,4}^\PV,$
the decorations of this correlator are denoted by $\ogamma=(\gb_1, \gb_2, \gb_3, \gb_4)$.
The moduli of $\Ga_{\bw}$-spin structure is $\cS_{0,4}(\ogamma)$. The boundary strata are
labeled by the following $G_\bw$-decorated dual graphs
\begin{figure}[H]
\centering
\begin{tikzpicture}[xscale=1,yscale=1]
\draw [-] (0, 0)--(north west:1);
\draw [-] (0, 0)--(south west:1);
\draw [-] (0, 0)--(.85, 0);
\draw [-] (1.15, 0)--(2, 0);
\draw [-] (2, 0)--+(north east:1);
\draw [-] (2, 0)--+(south east:1);
\node [ left] at (north west:1) {$\gb_1$};
\node [ left] at(south west:1) {$\gb_2$};
\node [ right] at (2.7, .7) {$\gb_3$};
\node [right] at (2.7, -.7) {$\gb_4$};
\node [above] at (.5, 0){$\gb_{1,+}$};
\node [above] at (1.5, 0){$\gb_{1,-}$};
\draw [fill=black] (0, 0) circle [radius = .05];
\draw [fill=black] (2, 0) circle [radius = .05];

\begin{scope}[shift ={(5, 0)}]
\draw [-] (0, 0)--(north west:1);
\draw [-] (0, 0)--(south west:1);
\draw [-] (0, 0)--(.85, 0);
\draw [-] (1.15, 0)--(2, 0);
\draw [-] (2, 0)--+(north east:1);
\draw [-] (2, 0)--+(south east:1);
\node [ left] at (north west:1) {$\gb_1$};
\node [ left] at(south west:1) {$\gb_3$};
\node [ right] at (2.7, .7) {$\gb_2$};
\node [right] at (2.7, -.7) {$\gb_4$};
\node [above] at (.5, 0){$\gb_{2,+}$};
\node [above] at (1.5, 0){$\gb_{2,-}$};
\draw [fill=black] (0, 0) circle [radius = .05];
\draw [fill=black] (2, 0) circle [radius = .05];
\end{scope}

\begin{scope}[shift ={(10, 0)}]
\draw [-] (0, 0)--(north west:1);
\draw [-] (0, 0)--(south west:1);
\draw [-] (0, 0)--(.85, 0);
\draw [-] (1.15, 0)--(2, 0);
\draw [-] (2, 0)--+(north east:1);
\draw [-] (2, 0)--+(south east:1);
\node [ left] at (north west:1) {$\gb_1$};
\node [ left] at(south west:1) {$\gb_4$};
\node [ right] at (2.7, .7) {$\gb_2$};
\node [right] at (2.7, -.7) {$\gb_3$};
\node [above] at (.5, 0){$\gb_{3,+}$};
\node [above] at (1.5, 0){$\gb_{3,-}$};
\draw [fill=black] (0, 0) circle [radius = .05];
\draw [fill=black] (2, 0) circle [radius = .05];
\end{scope}
\end{tikzpicture}
\caption{Boundary strata on $\cS_{0,4}(\ogamma)$}
\label{boundary}
\end{figure}

Each vertex represents a genus-zero component, and each half-edge represents a marking
(labeled with decorations $\gb_j$) or a node.
The decorations $\gb_{i,\pm}\in G_\bw$ on the same node are {\em balanced}, that is
$$\gb_{i,+}^{(j)}+\gb_{i,-}^{(j)}\equiv 0\mod 1.$$
According to the Selection Rule~\eqref{select},
decorations $\gb_{i,\pm}$ are determined by the other decorations on the same component.

For $j=1,\ldots,n$, we define {\em decoration vectors}
$$
v^{(j)}(\fF_n)=\left(\gb_1^{(j)}, \ldots, \gb_4^{(j)}\right), \quad v_+^{(j)} =
\left(\gb_{1, +}^{(j)}, \gb_{2, +}^{(j)}, \gb_{3, +}^{(j)}\right).
$$
For $j=n-1$ and $j=n$, we list these vectors non-concave correlators of type (b), (c), (d)
in Table~\ref{decoration-vector}.

\begin{table}[h]
  \centering
  \caption{Decoration vectors}
  \label{decoration-vector}
  \renewcommand{\arraystretch}{2}
 \begin{tabular}{|c|c|c|c|c|}
 \hline
Type
&$v^{(n-1)}(\fF_n)$
&$v_+^{(n-1)}$
&$v^{(n)}(\fF_n)$
&$v_+^{(n)}$
\\
\hline
(b)
&$(0,0,{3\over 2a_1-1}, {2a_1-2\over 2a_1-1})$
&$1-{1\over 2a_1-1}$
&$(0,0,{a_1-2\over 2a_1-1}, {a_1\over 2a_1-1})$
&$1-{1\over 2a_1-1}$
\\
\hline
(c)
&$(0,0,0,{2\over 3})$
&$({1\over 3}, {1\over 3}, {1\over 3})$
&$(0,0,0,{2\over 3})$
&$({1\over 3}, {1\over 3}, {1\over 3})$
\\
\hline
(d)
&$(0,0,{3\over 2a_{n-1}}, 1-{1\over 2a_{n-1}})$
&$({1\over 2a_{n-1}},{a_{n-1}-1\over a_{n-1}},{a_{n-1}-1\over a_{n-1}})$
&$(0,0,{1\over 2}, {1\over 2})$
&$({1\over 2}, 0, 0)$
\\\hline
\end{tabular}
\end{table}

The Type (a) case is more complicated, we only list
$$
v^{(n-1)}(\fF_n) = \left(\sum_{j=1}^{n-2}\rho_{j}^{(n-1)}, \sum_{j=1}^{n-2}\rho_{j}^{(n-1)},
  -3\rho_{n}^{(n-1)}+\sum_{j=1}^{n-2}\rho_{j}^{(n-1)}, 1-q_{n-1}\right)
$$
and
$$
v^{(n)}(\fF_n) = \left(\rho_n^{(n)}+\sum_{j=1}^{n}\rho_{j}^{(n)},
 \rho_n^{(n)}+\sum_{j=1}^{n}\rho_{j}^{(n)}, \sum_{j=1}^{n}\rho_{j}^{(n)}+\rho_{n-1}^{(n)},   1-q_n\right).
$$
We observe that in type (b) and (d), the first two insertions are broad, in type (c), the first
three insertions are broad, and there is no broad insertion in type (a).

\subsubsection{Gu\'er\'e's formula}
Let $\bw$ be an atomic invertible polynomial.
Let
$$Y=\langle v_1,\ldots,v_r\rangle_{g,r}^{\PV} =
\int_{\oM_{g,r}}\Lambda_{g,r}^\PV(\ogamma)(v_1\ot\ldots\ot v_r)
$$
be a correlator for the \PVs CohFT~\eqref{reduced-cohft}
decorated by an $r$-tuple $\ogamma=(\gb_1, \ldots, \gb_r)$.

Choose $v_i\in\cH_{\gb_i}$ from the standard basis. If $v_i$ is broad, then $v_i$ has one of the
forms in \eqref{chern-odd}, \eqref{chern-even}, \eqref{chern-chain}.
Following~\cite{Gue}, we call a variable $x_j$ {\em crossed to $v_i$}
if $\gb_i^{(j)}=0$, and $(-a_j x_j^{a_j-1})$ is not in \eqref{chern-odd}, \eqref{chern-even},
\eqref{chern-chain}.
For example, $x_j$ is crossed to $v_i=\Ch(K_{\rm odd})$ if $j$ is odd.
Define a line bundle $\cL_j^{\fC}:=\cL_j(-\sum p_i)$ by twisting the markings $p_i$ such that $x_j$
is crossed to $v_i$. The number of such markings with be denoted by $r_j$.

Let $t(j)$ be the unique subscript
such that $x_j^{a_{j}}x_{t(j)}$ is a monomial of $\bw$.
If a correlator contains at least one variable $x_j$ such that $H^0(\cC, \cL_j^{\fC})=0$,
then we can define
$$
\lambda_k=
\left\{
\begin{array}{ll}
 \lambda_j^{-a_j}, & \text{ if } k=t(j) \text{ and } H^0(\cC, \cL_j^{\fC})\neq0;\\
 \lambda, & \text{ otherwise}.
\end{array}
\right.
$$
For the correlator $Y$, denote the image of
$\phi_{g}(\ogamma)$ (cf.~\eqref{PV-class})
by $c^\PV_{\rm vir}(Y)$. If there is some $j$ such that $H^0(\cC, \cL_j^{\fC})=0$ for any
curve $\cC$ in  $\cS_{g}(\ogamma)$,
according to Gu\'er\'e's formula in~\cite[Theorem 3.21]{Gue} and the sign convention in
\eqref{reduced-cohft},
we have
\begin{equation}
\label{guere-formula}
-c^\PV_{\rm vir}(Y)=
\lim_{\lambda\to 1}\left(\prod_{j=1}^{n}(1-\lambda_j)^{-\Ch_0(R\pi_*\cL_j)+r_j}\right)
\exp\left(\sum_{j=1}^{n} \sum_{\ell\geq1} s_{\ell}(\lambda_j)\Ch_\ell(R\pi_*\cL_j)\right).
\end{equation}
Here $\Ch_\ell$ is the degree $\ell$ term of the Chern character, and
\[
s_\ell(x) = \frac{B_\ell(0)}{\ell} +(-1)^\ell\sum_{k=1}^\ell(k-1)!\left(\frac{x}{1-x} \right)^k\ga(\ell,k),
\]
where $B_{\ell}(t)$ is the $\ell$th
Bernoulli polynomial and $\ga(\ell,k)$ is
the coefficients of $z^{\ell}$ in the Taylor expansion of $\ell!(e^z-1)^k/k!$ at $z=0$.

Using the combinatorial preparation in Table \ref{decoration-vector}, we obtain
\begin{lem}
\label{nonconcave-variable}
Consider all the correlators $\fF_n$ in Table \ref{nonconcave}.
If $j\neq n-1$, then $H^0(\cC, \cL_{j}^{\fC})=0$  for any $\cC$.
If $j=n-1$, there exists some $\cC$, such that
$H^0(\cC, \cL_{n-1})\neq0.$
\end{lem}
\begin{proof}
 For $j<n-1$, by~\cite[Lemma 6.5]{HLSW},  $\deg\cL_j^{\fC}=\deg \cL_j=-1$ on each curve  $\cC$.
Thus $H^0(\cC, \cL_{j}^{\fC})=0$, and
\begin{equation}
\label{vanish-chern}
\Ch_1(R\pi_*\cL_j)=0, \quad j<n-1.
\end{equation}

Now we consider $j\geq n-1$.
Using the degree formula \eqref{deg-L-eq}, we see that on any smooth
$\cC$, $(\deg\cL_{n-1}, \deg\cL_{n})=(-1, 0)$ for type (a), (b), (d), and $(\deg\cL_{n-1},
\deg\cL_{n})=(0, 0)$ for type (c).
We can check that $(r_{n-1}, r_{n})=(0,2)$ for type (a), (b), (d) and $(r_{n-1}, r_{n})=(1,2)$ for
type (c).
For all the cases, we have
$\deg\cL_j^{\fC}=\deg\cL_j-r_j=-1-\de_{j}^n<0.$
Thus if $\cC$ is smooth, $H^0(\cC, \cL_j^{\fC})=0.$
Moreover,  we have a virtual degree formula
\begin{equation}
\label{virtual-recursive}
{\rm degvir} (\cL_j)
:=-\Ch_0(R\pi_*\cL_j)+r_j=\de_{j}^{n}.
\end{equation}

 If $\cC$ is singular, then it has a node $\cN$ and its normalization has two component,
 denoted by  $\cC_1$ and $\cC_2$.
 Using Table \ref{decoration-vector}, we can check that if $j\neq n-1$,
 the unordered pair $(\deg\cL_j^{\fC}|_{\cC_1}, \deg\cL_j^{\fC}|_{\cC_2})=(-2,-1)$ and the node
 $\cN$ is a narrow node with $H^0(\cN, \cL_j^{\fC}|_\cN)=0$,
 or $(\deg\cL_j^{\fC}|_{\cC_1}, \deg\cL_j^{\fC}|_{\cC_2})=(-2, 0)$
 and the node is broad with $H^0(\cN, \cL_j^{\fC}|_\cN)=\C$.
 In both cases, we have $H^0(\cC, \cL_j^{\fC})=0$, using the long exact sequence
 \begin{eqnarray*}
0\to H^0(\cC, \cL_j^{\fC})\to H^0(\cC_1, \cL_j^{\fC}|_{\cC_1})\oplus H^0(\cC_2,
   \cL_j^{\fC}|_{C_2})\to H^0(\cN, \cL_j^{\fC}|_\cN) \\
\to H^1(\cC, \cL_j^{\fC}) \to H^1(\cC_1, \cL_j^{\fC}|_{\cC_1})\oplus H^1(\cC_2, \cL_j^{\fC}|_{\cC_2})\to0
\end{eqnarray*}

 If $j\neq n-1$, the only exception happens when both $\th_n$ belong to the same component, say
 $\cC_1$, then  $(\deg\cL_j^{\fC}|_{\cC_1}, \deg\cL_j^{\fC}|_{\cC_2})=(0, -2)$ and the node $\cN$ is
 narrow.
 Then  $H^0(\cC, \cL_i^{\fC})\neq 0$.
\end{proof}

Using Lemma \ref{nonconcave-variable}, we can apply the formula \eqref{guere-formula} to the
correlator $\fF_n$ and get
\begin{lem}
For all the Type (ii) correlators $\fF_n$, the reduced CohFT is
\begin{equation}
\label{PV-4pt}
c_{\rm vir}^\PV(\fF_n)
=-a_{n-1}\Ch_1(R\pi_*\cL_{n-1})+\Ch_1(R\pi_*\cL_{n}).
\end{equation}
\end{lem}
\begin{proof}
By Lemma \ref{nonconcave-variable}, we see all the variables except $x_{n-1}$ are concave.
Thus
\begin{equation}
\label{lambda-value}
\lambda_{j}=
\left\{
\begin{array}{ll}
\lambda^{-a_{n-1}},
  & j=n; \\
\lambda, & j\neq n.
\end{array}
\right.
\end{equation}

Now the formula \eqref{PV-4pt} follows from the calculation
\begin{eqnarray*}
  -c_{\rm vir}^\PV(\fF_n)
  &=&\lim_{\lambda\to 1} (1-\lambda_n)\exp\left(\sum_{j=1}^n
      \sum_{\ell\geq1} s_{\ell}(\lambda_j)\Ch_\ell(R\pi_*\cL_j)\right)\\
&=&\lim_{\lambda\to 1} (1-\lambda_n)\left(1+\sum_{j=1}^{n}\left(-{1\over 2}-{\lambda_j\over
    1-\lambda_j}\right)\Ch_1(R\pi_*\cL_j)\right)\\
  &=&-\sum_{j=1}^n \lim_{\lambda\to 1}{\lambda_j(1-\lambda_n)\over 1-\lambda_j}\Ch_1(R\pi_*\cL_j)
\\
  &=&-\lim_{\lambda\to 1}{\lambda_{n-1}(1-\lambda_n) \over 1-\lambda_{n-1}}\Ch_1(R\pi_*\cL_{n-1}) -
      \lim_{\lambda\to 1}{\lambda_n(1-\lambda_n)\over 1-\lambda_n}\Ch_1(R\pi_*\cL_n)\\
  &=&a_{n-1}\Ch_1(R\pi_*\cL_{n-1})-\Ch_1(R\pi_*\cL_{n}).
\end{eqnarray*}
Here the first equality uses \eqref{guere-formula} and \eqref{virtual-recursive}, the fourth
equality uses \eqref{vanish-chern}, and the last equality uses  \eqref{lambda-value}.
\end{proof}

\subsubsection{Calculation of Chern characters}
Using Chiodo's formula~\cite[Theorem 1.1.1]{Chi}, we have
\begin{eqnarray*}
T_j&:=&{1\over \deg ({\rm st})} \cdot ({\rm st})_* \Ch_1(R\pi_*\cL_{j})\\
&=&\int_{\oM_{0,4}}
\left(\frac{B_{2}(q_j)}{2} \kappa_1 -\sum_{i=1}^4 \frac{B_{2}(\th^{(j)}_{\gb_i})}{2}\psi_i
    +\sum_{k=1}^{3}\frac{B_{2}(\th^{(j)}_{\gb_{k,+}})}{2}[\Ga_{k}]\right).
\end{eqnarray*}
Here $\kappa_1$ is the first kappa class and $\psi_i$ is the $i$-th psi class on $\oM_{0,4}$ and
$\Ga_1,\Ga_1, \Ga_1, $ are the $G_\bw$-decorated dual graphs of the boundary strata shown in
Figure~\ref{boundary}.

 By Table \ref{decoration-vector}, case (d),  the decoration vector for $\cL_n$ is
 $v^{(n)}(\fF_n)=(0,0,{1\over 2}, {1\over 2})$ and the  decoration vector $v_+^{(n)}=({1\over 2},  0,0)$.
 Since $q_n={1\over 2}$, we have
$$
T_n={1\over 2}\left(B_2(q_n)-2B_2(0)-2B_2\Big({1\over 2}\Big)+B_2\Big({1\over 2}\Big)+2B_2(0)\right)=0.$$
Other cases are computed similarly.
Integrating \eqref{PV-4pt},
we have
$$\fF_n=-a_{n-1}T_{n-1}+T_n.$$
We list the explicit calculations in Table \ref{table-calculation}.

\begin{table}[h]
\centering
\caption{Calculation of non-concave correlators}
\label{table-calculation}
\renewcommand{\arraystretch}{2}
 \begin{tabular}{|c|c|c|c|c|}
 \hline
Type
& Polynomial $\bw$
&$T_{n-1}$
&$T_n$
&$\fF_n$
\\
\hline
(a)
&$\sum\limits_{i=1}^{n-1}x_i^{a_i}x_{i+1}+x_n^2x_1, n\geq3$
&$-q_{n-1}-2\rho_{n}^{(n-1)}$
&$-1+2\rho_n^{(n)}$
&$-q_n$
\\
\hline
(b)
&$x_1^{a_1}x_2+x_2^2x_1, a_1\geq 3$
&${1\over 2a_1-1}$
&${1\over 2a_1-1}$
&$-{a_1-1\over 2a_1-1}$
\\
\hline
(c)
&$x_1^{2}x_2+x_2^2x_1$
&${1\over 3}$
&${1\over 3}$
&$-{1\over 3}$
\\
\hline
(d)
&$\sum\limits_{i=1}^{n-1}x_i^{a_i}x_{i+1}+x_n^2$
&${1\over 2a_{n-1}}$
&$0$
&$-{1\over 2}$
\\\hline
\end{tabular}
\end{table}
Now Proposition \ref{mirror-invariant-prop} is proved.
We remark that case (a) is computed in~\cite[Appendix]{HLSW}, where
$$\fF_n=-a_{n-1}(-q_{n-1}-2\rho_{n}^{(n-1)})+(-1+2\rho_n^{(n)})=-q_n.$$

\appendix
\section{Reconstruction for polynomials of the chain type}
\label{sec-appendix}
In this section we present a proof of the reconstruction result
Proposition~\ref{reconstruction} for a polynomial of the chain type,
$\bw=\sum\limits_{i=1}^{n-1}x_i^{a_i}x_{i+1}+x_n^{a_n}$.

Consider a nonzero correlator~\eqref{C2}
$$
X=\langle \underbrace{\th_n, \ldots, \th_n}_{\ell_n}, \underbrace{\th_{n-1}, \ldots,
  \th_{n-1}}_{\ell_{n-1}}, \ldots, \underbrace{\th_1, \ldots, \th_1}_{\ell_1}, \alpha, \beta
\rangle,
$$
with $$\alpha=\prod_{i=1}^{n} \th_i^{p_i}, \  \beta=\prod_{i=1}^{n} \th_i^{q_i}\in \cHw,$$
as in~\eqref{basis-assumption} (except that here we are using lower-case letters $p_i$ and
$q_i$ which should not lead to a confusion since the weights of the variables do not
appear in the appendix).
Note that from the description of the standard basis in
Table~\ref{table-basis}, it follows that
\begin{equation}
  \label{eq:ineq}
p_i\leq a_i-1, \ q_i\leq a_i-1, \text{\ and \ }
p_i+q_i\leq 2a_i-2, \text{\ \ for\ \ } i=1,\ldots,n.
\end{equation}

Together with formula~\eqref{chain-entry} for $\rho_j^{(i)}$ and the fact that
$\ell_i\geq 0$,
this gives the following constraints between the integers $K_1, \ldots, K_n$ defined
by~\eqref{b_def} and~\eqref{ind-k}.
\begin{align}\label{mplusn_eq}
  p_i+q_i&=a_i(\ell_i-K_i+1)+(\ell_{i+1}-K_{i+1}+1)-\ell_i-2,\\
\label{pq_n} p_n+q_n&=a_n(\ell_n-K_n+1)-\ell_n-2,\\
\label{eq2}   a_iK_i+K_{i+1}&\geq (a_i-1)(\ell_i-1)+\ell_{i+1},\\
  \label{k_n}a_nK_n&\geq (a_n-1)(\ell_n-1)-1,\\
\label{key_equation}
  K_i+K_{i+1}&\geq (1-a_i)(1+K_i),
\end{align}
for all $i=1,\ldots,n-1$.
These equations further imply the following additional relations.
\begin{lem}\label{key-corollary}
We have
  \begin{itemize}
  \item
If $K_i < 0$, \ for some $i\leq n-1$, \ then $K_i+K_{i+1}\geq 0$.
\item If $K_i<0$ and $K_i+K_{i+1}= 0$, then $(K_i, K_{i+1})=(-1, 1)$. \\
  In this case  $\ell_i=\ell_{i+1}=0$,
$p_i+q_i=2a_i-2$, and $p_{i+1}+q_{i+1}=\ell_{i+2}-K_{i+2}-1$.

\item    $-1\leq K_n \leq \ell_n$ and, if $K_n=-1$, then $\ell_n=0$, $p_n+q_n=2a_n-2$.
\end{itemize}
  \end{lem}
This leads to a complete description of possible collections  $(K_1, \ldots, K_n)$ with  $K_n\geq 0$.
\begin{lem}\label{loop_condition_lemma}
If $K_n\geq 0$, then the tuple $(K_1, \ldots, K_n)$ is of one of the following kinds:
\begin{itemize}
\item A concatenation of some $(0)$'s and $(-1, 1)$'s with one $(1)$, in any order as long as it
  does not end with $(-1, 1)$;
 \item A concatenation of some $(0)$'s and $(-1, 1)$'s with one of $(1)$, $(-1, 2)$, or $(-2, 3)$
 ending with two non-negative numbers.
\end{itemize}
\end{lem}
\begin{proof}
First, we observe that $K_{n-1}\geq 0$. Indeed, the assumption $K_{n-1}\leq -1$ together with
$\ell_n\geq K_n$ contradicts to \eqref{eq2}.

Now, Lemma \ref{key-corollary} implies that if $K_i<0$ for some $i< n-1$, then $K_{i-1}\geq 0$
and $K_i+K_{i+1}\geq 0$.
Since, by Lemma~\ref{lem:sumK}, we have $K_1+\ldots+K_n=1$,  removing all pairs
$(K_i, K_{i+1})$ with $K_i<0$, will leave us with a tuple of non-negative integers
$(\wt{K_1}, \ldots, \wt{K_s})$ (recall that $K_n\geq 0$) such that $\wt{K_1}+\ldots+\wt{K_s}\leq 1$.
Therefore at most one of them can be nonzero. Also, since the removed pairs
$(K_i,K_{i+1})$ satisfy $0\le K_i+K_{i+1}\leq 1$, for all but at most one of these pairs we
have $K_i+K_{i+1}=0$, and so $(K_i, K_{i+1})= (-1, 1)$ by Lemma \ref{key-corollary}.
If $K_i+K_{i+1}=1$, then from \eqref{key_equation} it follows that $(K_i,K_{i+1})$ must be
$(-1, 2)$ or $(-2, 3)$.
\end{proof}

The correlator $X$ in \eqref{C2} can be reconstructed with the correlators of the
following much simpler form.
\begin{prop}\label{chain recon}
We can reconstruct correlators in \eqref{C2} from correlators of the form
\begin{equation}\label{C3}
  X=\<\underbrace{\th_n, \ldots, \th_n}_{\ell_n\ \text{copies}}, \alpha, \beta\>.
\end{equation}
\end{prop}
\begin{proof}
Starting with a correlator $X$ in \eqref{C2} not in \eqref{C3}, we can choose $i$ to be
the largest index with $i < n$ and $\ell_i \geq 1$.
More precisely, $X$ is of the following form:
$$X=\<\th_i, \bth_S, \alpha, \beta\>, \, i<n,$$
where $\bth_S$
is a tuple consisting of $\th_j$s with $j=n$ or $j\leq i$.
Now it is sufficient to prove that using
\eqref{reconst_eq},  $X$ can be reconstructed from correlators with fewer insertions
and correlators of the form
  $$Z=\<\th_j, \th_S, \alpha^\prime, \beta^\prime\>, \, j>i.$$
Here the set $\bth_S$ in $Z$ is the same as that in $X$, but $\alpha^\prime$,
$\beta^\prime$ can be different form $\alpha$ and $\beta$ in $X$.

Notice that
        $X = \< \th_S, \th_i, \th_n^{p_n}\widetilde{\alpha}, \beta \>$ for some $p_n\geq0$.
If $p_n\geq1,$ then we apply \eqref{reconst_eq} with $\ga = \beta$, $\de = \th_i$,
$\epsilon = \th_n$, and $\phi = \th_n^{p_n-1} \widetilde{\alpha}$.
The correlators with $\de \phi$ and $\de \ga$ are of the form:
$$
\<\th_S, \th_n,  \th_i\th_n^{p_n-1} \widetilde{\alpha}, \beta\>,\quad \<\th_S, \th_n,  \th_n^{p_n-1}
\widetilde{\alpha}, \th_i\beta\>.
$$
They are both of the form $Z$. The correlator with $\epsilon  \ga$ equals $\langle\th_S,
\th_i, \th_n^{p_n-1} \alpha, \th_n  \beta \rangle$.
By induction we can reconstruct $X$ from $Z$ and the correlator $Y = \langle  \th_S,
\th_i, \alpha_Y, \beta_Y \rangle$ where $p^Y_n=0$.

Similarly, we move all $x_{n-1}$ from $\alpha_Y$ to $\beta_Y$, and so on, until we move
all $\th_{i+1}$ from $\alpha$ to $\beta$.
Thus we reconstruct $X$ from correlators $Z$, and the correlator $Y = \langle\ldots,
\th_i, \alpha_Y, \beta_Y \rangle$ where $p^Y_{i+1} = \ldots = p^Y_n = 0$.

After reducing to the basis listed in Table \ref{table-basis}, $Y$ satisfies $p^Y_{k} +
q^Y_{k} \leq a_k-1$ for $k> i$. By Lemma \ref{key-corollary}, we have $K_n^Y\geq 0$ in
$Y$. In the following argument, we focus on the reconstruction of $Y$, and drop the
superscript $Y$ on $K$, $\bp$ and $\bq$.

\

\noindent{\bf{Case $K_n=1$:}}
In this case $\bK$ is a concatenation of (0)s and (-1,1)s, followed by $K_n=1$.

If $\bK = (\ldots, -1,1,1)$, then $\bl =( \ldots, 0, 0, *)$ and
$\bp + \bq = (\ldots, 2 a_{n-2}-2, *, *)$
by \eqref{mplusn_eq}.
Then $n-2>i$, but $p_{n-2} + q_{n-2} \geq a_{n-2}$, contradicting our assumption $p_{k} + q_{k}\leq a_k-1$  on $Y$.
Similarly, we reach a contradiction if there is $j > i$ such that $(K_j, K_{j+1}) = (-1,1)$.

Therefore, $\bK = ( \ldots, \underline{0},0, \ldots, 0,1)$ and $\bl = (\ldots, \underline{1}, 0, \ldots,0,*)$,
where the underline marks the $i^{th}$ spot and $\ell_i=1$ by \eqref{eq2}.
Possibly, $i=n-1$.
If $i \neq n-1$, then by assumption $(K_{i+1}, \ell_{i+1}) = (0,0)$ so $p_i+q_i = 2 a_i-2$ by \eqref{mplusn_eq}.
If $i=n-1$, then $(K_{i+1}, \ell_{i+1}) = (K_n,\ell_n)$ where $\ell_n \geq K_n$. Then
(\ref{mplusn_eq}) shows $p_i + q_i \geq 2 a_i-2$ so by
\eqref{eq:ineq} we know that
$p_i + q_i = 2 a_i-2$.
Thus
$\bp + \bq = ( \ldots, \underline{2 a_i-2}, *, \ldots, *, *).$
Now we have three cases. In each case we compute $\bp+\bq$ by first using \eqref{eq2} to
compute $\ell$ and then using (\ref{mplusn_eq}),
\eqref{basis-assumption}, and Lemma~\ref{key-corollary}.
\begin{enumerate}
\item
$\bK=(  \ldots, 0, \underline{0}, \ldots, 1),$
$\bp+\bq=(\ldots, a_{i-1}, \underline{M}, \ldots, *).$
\item
$\bK=(  -1, 1, \ldots, -1, 1, \underline{0}, \ldots, 1),$
$\bp+\bq=(M, 0, \ldots, M, 0, \underline{M}, \ldots, *).$
\item
$\bK=( \ldots, 0,-1, 1, \ldots, -1, 1, \underline{0}, \ldots, 1),$
$\bp+\bq=( \ldots, a_r, M, 0, \ldots, M, 0, \underline{M}, \ldots,* ).$
\end{enumerate}
Here  $M = 2 a-2$ with the appropriate subscripts.
In each case, we claim that there is $\widehat{\alpha}\in \cHw$ that satisfying
$\alpha=\th_{i+1}^{a_{i+1}}\bullet \widehat{\alpha}$.
We find a factor of $\th_{i+1}^{a_{i+1}}$ in $\alpha$ for each case as follows:
\begin{enumerate}
\item Here $\alpha$ has a factor of $\th_{i-1} \th_{i}^{a_i-1}$, which by \eqref{loop-relation}
is proportional to
  $\th_{i+1}^{a_{i+1}}$ in $\cHw$.
\item Repeatedly apply \eqref{loop-relation} starting with $a_1 \th_1^{a_1-1} = - \th_2^{a_2}$.
\item Repeatedly apply \eqref{loop-relation} starting with $a_{r+1} \th_r
  \th_{r+1}^{a_{r+1}-1} = - \th_{r+2}^{a_{r+2}}$.
\end{enumerate}

Now apply \eqref{reconst_eq} to $Y$ with
$\ga = \th_i$, $\de = \beta$, $\epsilon =
\th_{i+1}^{a_{i+1}}$, and $\phi = \widehat{\alpha}$.
Then $\ga  \epsilon=\th_i\th_{i+1}^{a_{i+1}}$ vanishes by
Corollary~\ref{vanishing}
and the other two correlators have the form
$\langle \bth_S, \th_{i+1}^{a_{i+1}}, \alpha', \beta' \rangle$.
Writing
$\th_{i+1}^{a_{i+1}}=\th_{i+1}\bullet\th_{i+1}^{a_{i+1}-1}$, and
performing reconstruction scheme similar to the one in the proof of Lemma \ref{Step 1},
we will obtain the correlator of the required form.

\

\noindent{\bf{Case $K_n=0$.}} \
In this case, \eqref{k_n} implies that there are three possibilities:
$\ell_n=0$; $\ell_n=1$; or $\ell_n = 2, a_n=2$.
Using~\eqref{pq_n} we see that the cases $\ell_n=a_n=2$ and $\ell_n=1, a_n\geq 3$
contradict our assumption that $p_n+q_n\leq a_n-1$. So it only remains to consider the
cases $\ell_n=0$ and $\ell_n=1$, $a_n=2$.

Let us first consider the case $\ell_n=1$, $a_n=2$. By \eqref{pq_n}, we have
$p_n+q_n=1=a_n-1$. If $p_{n-1}+q_{n-1}> 0$, we assume without loss of generality that
$p_n=1$ and $q_{n-1}>0$ and write $\alpha=\th_n\bullet \alpha'$.
Applying \eqref{reconst_eq} to $Y$ with $\ga = \th_i$, $\de = \beta$, $\epsilon = \th_n$
and $\phi = \alpha'$, we can obtain the required correlator, since $\th_{n-1}\th_n=0$.
If $p_{n-1}+q_{n-1}=0$, equations \eqref{mplusn_eq} and \eqref{eq2}
imply that $K_{n-1}=1$ and $\ell_{n-1}=0$.
Let $i<n-1$  be the largest subscript such that $\ell_i\neq 0$.

There are three cases:%
\footnote{In \cite{HLSW}, these three cases have not been considered.
However, this is not a serious issue, since the main result of \cite{HLSW} does not
include the case $a_n=2$.}
\begin{enumerate}
\item
$\bK=(  \ldots, 0, \underline{0}, \ldots, 1, 0),$
$\bp+\bq=(\ldots, a_{i-1}, \underline{M}, \ldots, 0, 1).$
\item
$\bK=(  -1, 1, \ldots, -1, 1, \underline{0}, \ldots, 1, 0),$
$\bp+\bq=(M, 0, \ldots, M, 0, \underline{M}, \ldots, 0, 1).$
\item
$\bK=( \ldots, 0,-1, 1, \ldots, -1, 1, \underline{0}, \ldots, 1, 0),$
$\bp+\bq=( \ldots, a_r, M, 0, \ldots, M, 0, \underline{M}, \ldots, 0, 1).$
\end{enumerate}
The discussion is similar to the case $K_n=1$.

Now assume $\ell_n=0$, and let $i$ be the largest subscript
such that $(K_i, \ell_i) \neq (0,0)$.
Since $0 \leq p_i + q_i$, equation (\ref{mplusn_eq}) shows $K_i \leq \ell_i$.
 Then \eqref{eq2} shows that $(K_{i-1}, K_i)$ cannot be $(-2,3)$, $(-1,2)$, or $(-1,1)$.
Six cases remain, and the reconstruction can be completed using the strategy analogous to the $K_n=1$ case
 (or see \cite{HLSW} page 47).
\end{proof}

Furthermore, we have
\begin{lem}
We can reconstruct correlators in \eqref{C3} from correlators of the form
\begin{equation}\label{C4}
  X=\<\th_n, \th_n, \alpha, \beta\>.
\end{equation}
\end{lem}
\begin{proof}
Now we focus on correlator $X$ in \eqref{C3}.
From (\ref{k_n}), since $\ell_n \geq 2$, we find
\[
K_n \geq {a_n-2\over a_n}.
\]
Thus $K_n \geq 0$ and equality is possible only if $a_n=2$.

If $K_n=0$ and $a_n=2$, then \eqref{k_n} shows that $\ell_n\leq 2$.

If $K_n \neq 0$, then $K_n=1$ by Lemma \ref{loop_condition_lemma}.
Then (\ref{k_n}) shows $\ell_n \leq 2a_n/(a_n-1)$, so $\ell_N=2$ , or $\ell_n=a_n=3$, or
$a_n=2$ and $\ell_n=3$ or $4$.
We will show that in each case where $\ell_n>2$, the correlator does not satisfy
\eqref{basis-assumption}, a contradiction.

If $\ell_n=a_n=3$, then $p_n + q_n= 3 a_n - 5 = 2 a_n-2$.
Either $(K_{n-1}, \ell_{n-1})$ is $(0,0)$ or it is $(1,0)$; in each case, $p_{n-1} + q_{n-1} \geq 1$.
Without loss of generality $p_{n-1} \geq 1$, so that $\alpha$ has a factor of $\th_{n-1}
\th_n^{a_n-1}$, violating \eqref{basis-assumption}.

Similarly, if $a_n=2$ and $\ell_n=3$ or 4, we can check all possibilities for $\bK$ and
$\bl$ and show that $\bp+\bq$ violates \eqref{basis-assumption}.
\end{proof}

Finally, we only need to show
\begin{lem}
We can reconstruct correlators in \eqref{C4} from correlators of the form
\begin{equation}
  X=\<\th_n, \th_n, \th_{n-1}\th_n^{a_n-2}, \th_{J^{-1}}\>
\end{equation}
\end{lem}
\begin{proof}
Let $X$ be a correlator in \eqref{C4}
We know $\bl = (0, \ldots, 0, 2)$. By \eqref{mplusn_eq}-\eqref{key_equation}, if $M =
2a-2$, we have three possibilities for $\bK$:
\begin{enumerate}
\item
$\bK = (0, \ldots, 0, 0, 1),$ $\bp+\bq=(a_1-1, \ldots, a_{n-2}-1, a_{n-1}, 2a_n-4)$.

\item
$\bK=(-1,1, \ldots, -1,1, 1),$ $\bp+\bq = (M, 0, \ldots, M, 0, 2a_n-4)$.

\item
$\bK=(0, \ldots, 0, 0, -1, 1, \ldots, -1, 1, 1)$, $\bp+\bq = (a_1-1, \ldots, a_{r-1}-1,
a_r, M, 0, \ldots,  M,$ $0, 2a_n-4)$.
\end{enumerate}

In all cases, if $X\neq0$, we must have
\begin{equation}\label{chain-middle}
X = \< \th_n, \th_n, \th_1^{p_1} \ldots \th_{n-1}^{p_{n-1}}\th_n^{a_n-2}, \th_1^{q_1}
\ldots \th_{n-1}^{q_{n-1}} \th_n^{a_n-2}\>
\end{equation}
where $p_i + q_i = a_i-1$ for $i\leq n-2$ and $p_{n-1} + q_{n-1} = a_{n-1}$.

In the first case, both $p_{n-1}$ and $q_{n-1}$ are at least 1. If $p_{n} = a_n-1$, then
$X=0$ by \eqref{chain-last}, since it has a factor of $\th_{n-1}\th_n^{a_n-1}$. This shows
that $p_n = q_n = a_n-2$ and \eqref{chain-middle} follows.

In the second case, $\alpha = \th_1^{a_1-1} \th_3^{a_3-1} \ldots \th_{n-2}^{a_{n-2}-1} \th_n^{p_n}$.
The relations \eqref{loop-relation} show
$$\alpha \propto \th_2^{a_2-1}\th_4^{a_4-1} \ldots \th_{n-1}^{a_{n-1}} \th_n^{p_n}.$$
If $p_n = a_n-1$, we have a factor of $\alpha$ equal to $\th_{n-1} \th_n^{a_n-1}$, and
$\alpha=0$ by \eqref{chain-last}.
Otherwise, $p_n=q_n=a_n-2$ and \eqref{chain-middle} follows.

In the last case,  $\alpha$ has a factor equal to $\th_{r}^{p_r} \th_{r+1}^{a_{r+1}-1}
\ldots \th_{n-2}^{a_{n-2}-1} \th_n^{p_n}$. As before we use the relations
\eqref{loop-relation} to rewrite it as $\th_r^{p_r-1}\th_{r+2}^{a_{r+2}-1} \ldots
\th_{n-1}^{a_{n-1}} \th_n^{p_n}$. As before, if $X\neq0$, then \eqref{chain-middle}
follows.

Finally, we apply \eqref{reconst_eq} to $X$ in \eqref{chain-middle} with $\ga = \th_n$,
$\epsilon = \th_{n-1}\th_n^{a_n-2}$, $\phi = \th_1^{p_1} \ldots \th_{n-1}^{p_{n-1}-1} $,
and $\de =  \th_1^{q_1} \ldots \th_{n-1}^{q_{n-1}} \th_n^{a_n-2} $. Then $\epsilon  \ga $
and $\ga  \de$ have a factor of $\th_{n-1} \th_n^{a_n-1}$, and hence both are 0 by
\eqref{chain-last}.
The remaining term containing $\de\phi$ is exactly $\fF_n$.
\end{proof}

\noindent{\small Department of Mathematics, Sun Yat-sen University, Guangzhou, Guangdong 510275, China}
\quad
\noindent{\small E-mail: \tt hewq@mail2.sysu.edu.cn}

\noindent{\small Department of Mathematics, University of Oregon, Eugene, OR 97403, USA;
 National Research University Higher School of Economics; and Korea Institute for  Advanced Study}

\noindent{\small E-mail: \tt apolish@uoregon.edu}

\noindent{\small Department of Mathematics, University of Oregon, Eugene, OR 97403, USA}

\noindent{\small E-mail: \tt yfshen@uoregon.edu}

\noindent{\small Department of Mathematics, University of Oregon, Eugene, OR 97403, USA}

\noindent{\small E-mail: \tt vaintrob@uoregon.edu}

\end{document}